\documentclass[reqno,a4paper, 11pt]{amsart}

\usepackage[pdfpagelabels]{hyperref}
\usepackage[english]{babel}
\usepackage[utf8]{inputenc}
\usepackage{amsfonts,epsfig}
\usepackage{latexsym}
\usepackage{amsmath}
\usepackage{amssymb}
\usepackage{mathtools}
\usepackage{pgfplots}
\pgfplotsset{compat=newest}
\usepackage{standalone}
\usepackage{mathabx}
\usepackage{comment}
\usepackage{color}
\usepackage{enumerate}
\usepackage{csquotes}
\usepackage[backend=biber,
            maxnames=100,
            style=alphabetic,
            giveninits=true,
            isbn=false,
            doi=false,
            url=false,
            sorting=nyvt]{biblatex}
\renewbibmacro{in:}{}

\newtheorem{theorem}{Theorem}
\newtheorem{lemma}[theorem]{Lemma}

\newtheorem{proposition}[theorem]{Proposition}

\theoremstyle{definition}

\theoremstyle{remark}

\numberwithin{equation}{section}

\newcommand{\set}[1]{\left\{{#1}\right\}}
\newcommand{\measure}[1]{m \left({#1}\right)}
\newcommand{\measset}[1]{m \left(\set{#1}\right)}
\newcommand{\norm}[2]{\left\|{#1}\right\|_{{#2}}}
\newcommand{\Lp}[1]{\mathrm{L}^{#1}}
\newcommand{\B}{\mathcal{B}}
\newcommand{\D}{\mathbb{D}}

\newcommand{\N}{\mathbb{N}}
\newcommand{\R}{\mathbb{R}}
\newcommand{\C}{\mathbb{C}}
\newcommand{\A}{\mathcal{A}}
\newcommand{\J}{\mathcal{J}}
\DeclareMathOperator{\Real}{Re}
\def\circle{\partial \D}
\def\ic{\int_{\partial \D}}

\title{The Central Limit Theorem for inner functions II}

\author{Artur Nicolau}

\author{Od\'i Soler i Gibert}

\newcommand{\Addresses}{{% additional braces for segregating \footnotesize
  \bigskip
  \footnotesize

  \noindent Artur~Nicolau\\
  \textsc{Universitat Aut\`onoma de Barcelona and Centre de Recerca Matem\`atica\\
          Departament de Matem\`atiques\\
          08193 Bellaterra (Cerdanyola del Vallès)\\
          Catalunya}\\
  \textit{E-mail address: }\texttt{artur.nicolau@uab.cat}

  \bigskip

  \noindent Odí~Soler~i~Gibert\\
  \textsc{Universitat Politècnica de Catalunya - BarcelonaTech (UPC),\\
          Pavelló I,\\
          Diagonal 647,\\
          08028 Barcelona\\
          Catalunya}\\
  \textit{E-mail address: }\texttt{odi.soler@upc.edu}

}}

\thanks{The authors are supported in part by the Generalitat de Catalunya (grant 2021 SGR 00071), the Spanish Ministerio de Ciencia e Innovaci\'on (project  PID2021-123151NB-I00) and the Spanish Research Agency (Mar\'ia de
Maeztu Program CEX2020-001084-M).
        The second author was also partially supported by the ERC project CHRiSHarMa no. DLV-682402.}
        
\begin{document}

    \begin{abstract}
        A sharp version of the Central Limit Theorem for linear combinations of iterates of an inner function is proved. The authors previously showed this result assuming a suboptimal condition on the coefficients of the linear combination.
        Here we explain a variation of the argument which leads to the sharp result.
        We also review the steps of the proof as well as the main technical tool which is Aleksandrov Disintegration Theorem of Aleksandrov-Clark measures.
    \end{abstract}
    
    \maketitle

    \section{Introduction and main results}
    \label{sec:Introduction}
    Let $\D$ be the unit disc in the complex plane and let $m$ be the normalized Lebesgue measure on the unit circle $\circle$. An analytic mapping from $\D$ into $\D$ is called inner if its radial limits have modulus one at almost every point of the unit circle. Hence any inner function $f$ induces a map defined at almost every point $z \in \circle$ as $f(z) = \lim_{r \to 1^{-}} f(rz)$. Let $f^n = f \circ \ldots \circ f : \circle \rightarrow \circle$ denote the $n$-th iterate of the inner function $f$. It has been recently shown that the iterates $\{ f^n \}$ behave as a sequence of independent random variables in the sense that they obey appropriate versions of classical results on sequences of independent random variables (see \cite{ref:Nicolau2022} and \cite{ref:NicolauSolerGibert2022}). There are of course other situations in Complex Analysis where one encounters probabilistic behaviors in settings where the notion of independence is not directly present. A classical example of such setting is the assymptotic behaviour of lacunary series. In a series of classical papers by Paley and Zygmund, Salem and Zygmund and Weiss, the authors consider versions of the Khintchine-Kolmogorov Theorem for pointwise convergence (see~\cite[Section~V.6]{ref:ZygmundTrigonometricSeries-SecondEd}), versions of the Central Limit Theorem (\cite{ref:SalemZygmundLacunarySeries-I} and \cite{ref:SalemZygmundLacunarySeries-II}) and of the Law of the Iterated Logarithm (\cite{ref:WeissLILLacunarySeries}), for lacunary series. Our work is inspired by the Central Limit Theorem for lacunary series proved by Salem and Zygmund.

    Recall that a sequence of measurable functions $\{f_N\}$ defined
    at almost every point in the unit circle converges in distribution
    to a (circularly symmetric) standard complex normal variable if for any Borel set $K \subset \C$ such that its boundary $\partial K$ has zero area, one has 
    \begin{equation*}
        \lim_{N \to \infty} \measset{z \in \circle\colon f_N (z) \in K}
        = \frac{1}{ 2 \pi} \int_K e^{- |w|^2 / 2}\, dA(w) . 
    \end{equation*} 
    Let $f$ be an inner function with $f(0)=0$ which is not a rotation and let $\{a_n\}$ be a sequence of complex numbers. A version of the Central Limit Theorem for linear combinations $\sum a_n f^n$ of iterates has been given in ~\cite{ref:NicolauSolerGibert2022} under certain conditions on the size of the coefficients $\{a_n\}$. The main purpose of this paper is to present the following version of this result which holds under the minimal assumption on the size of the coefficients $\{a_n \}$. 
    
    \begin{theorem}
        \label{thm:InnerFunctionsCLT}
        Let $f$ be an inner function with $f(0) = 0$ which is not a rotation.
        Let $\{a_n\}$ be a sequence of complex numbers with
        \begin{equation}
            \label{eq:CoefficientsL2Divergence}
             \sum_{n=1}^\infty |a_n|^2 = \infty.
        \end{equation}
        Consider 
        \begin{equation}
            \label{eq:PartialSumVariance}
            \sigma_N^2
            = \sum_{n=1}^{N} |a_n|^2 + 2 \Real \sum_{k=1}^{N-1} f'(0)^k \sum_{n=1}^{N-k} \overline{a_n} a_{n+k},
            \quad  N=1,2,\ldots
        \end{equation}
        Assume that the coefficients $\{a_n\}$ satisfy that 
        \begin{equation}
            \label{eq:CoefficientGrowth}
            \lim_{N \to \infty}  \frac{|a_N|^2}{\sum_{n=1}^N |a_n|^2} = 0. 
        \end{equation}
        Then 
        \begin{equation*}
            \frac{\sqrt{2}}{\sigma_N} \sum_{n=1}^{N} a_n f^n
        \end{equation*}
        converges in distribution to a standard complex normal variable.  
    \end{theorem}

   % Theorem~\ref{thm:InnerFunctionsCLT} was recently proved by the authors~\cite{ref:NicolauSolerGibert2022} under a more restrictive hypothesis on the coefficients $\{a_n \}$ . 
   Other versions of the Central Limit Theorem in this context have been given in \cite{ivrii2023inner} and \cite{aaronson2023dynamics}. 
   
   We will explain an argument due to Salem and Zygmund to show that condition~\eqref{eq:CoefficientGrowth} is optimal in the sense that
   for any sequence $\{a_n\}$ with $\sum |a_n|^2 = \infty$ and  
    \begin{equation*}
        \limsup_{N \to \infty}  \frac{|a_N|^2}{\sum_{n=1}^N |a_n|^2} > 0,
    \end{equation*}
    there are examples of inner functions for which the conclusion cannot hold.
    For $1\leq p < \infty$ let $\|g\|_p$ denote the $\Lp{p}$ norm of the function $g$ on the unit circle defined by 
    \begin{equation*}
        \|g \|_p^p = \int_{\circle} |g|^p dm. 
    \end{equation*}
     Let $f$ be an inner function with $f(0) = 0$ which is not a rotation and let $\{a_n\}$ be a sequence of complex numbers. It was proved in~\cite{ref:NicolauSolerGibert2022} that
    \begin{equation}
        \label{eq:SNSigmaNComparability}
          \frac{1-|f'(0)|}{1+|f'(0)|} \sum_{n=k}^N |a_n|^2 \leq \sigma_N^2 = \norm{\sum_{n=k}^N a_n f^n}{2}^2 \leq \frac{1+|f'(0)|}{1-|f'(0)|} \sum_{n=k}^N |a_n|^2 , \quad N= 1,2,\ldots,
    \end{equation}
    for $1 \leq k \leq N.$
    Hence the series $\sum a_n f^n$ converges in $\Lp{2} (\circle)$ if and only if $\sum |a_n|^2 < \infty$. Moreover if this last condition holds, then the series $\sum a_n f^n (z)$ converges at almost every point $z \in \circle$ (see \cite{ref:Nicolau2022}). In this context, repeating the proof of Theorem~\ref{thm:InnerFunctionsCLT}, one can show that when one has pointwise convergence, the tails obey a Central Limit Theorem.  

  \begin{theorem}
        \label{thm:TailsCLT}
        Let $f$ be an inner function with $f(0) = 0$ which is not a rotation.
        Let $\{a_n\}$ be a sequence of complex numbers with $\sum |a_n|^2 < \infty$.
        Consider 
        \begin{equation*}
            \sigma^2 (N)
            = \sum_{n \geq N} |a_n|^2 + 2 \Real \sum_{k \geq 1} f'(0)^k \sum_{n \geq N} \overline{a_n} a_{n+k},
            \quad  N=1,2,\ldots
        \end{equation*}
        Assume that the coefficients $\{a_n\}$ satisfy that
        \begin{equation}
            \label{eq:TailsCoefficientDecay}
            \lim_{N \to \infty}  \frac{|a_N|^2}{\sum_{n \geq N}  |a_n|^2} = 0.
        \end{equation}
        Then 
        \begin{equation*}
            \frac{\sqrt{2}}{\sigma (N)} \sum_{n=N}^{\infty} a_n f^n
        \end{equation*}
        converges in distribution to a standard complex normal variable.
    \end{theorem}

       Given a set $\A$ of positive integers, consider the corresponding partial sum 
    \begin{equation*}
        \xi (\A) = \sum_{n \in \A} a_n f^n.
    \end{equation*} 
    The proof of Theorem~\ref{thm:InnerFunctionsCLT} contains two main ideas. The first one is a convenient splitting of the partial sums
    \begin{equation*}
        \sum_{n=1}^N a_n f^n = \sum_k (\xi (\A_k) + \xi (\B_k) )
    \end{equation*}
    into certain alternating blocks of consecutive terms $\xi (\A_k)$ and $\xi (\B_k)$ which depend on $N,$ that is, satisfying $\max \A_k + 1 = \min \B_k$ and $\max \B_k + 1= \min \A_{k+1} $, which obey two counteracting properties. On the one hand $\|\sum \xi (\B_k) \|_2 / \sigma_N$ must be small so that $\sum \xi (\B_k)$ becomes irrelevant. On the other hand the number of terms of each block $\xi(\B_k)$ must be large so that the correlations between different blocks $\xi (\A_k)$ decay sufficiently fast. The construction of these blocks is inspired by a similar construction of Weiss in \cite{ref:WeissLILLacunarySeries}. The second main idea in the proof was already present in~\cite{ref:NicolauSolerGibert2022} and concerns the decay of certain correlations which naturally appear when proving a Central Limit Theorem. We mention two main properties. 
    If $\A$ and $\B$ are two sets of positive integers such that $a<b$ for any $ a \in \A$ and $b \in B$, it turns out that $|\xi(\A)|^2$ and $|\xi(\B)|^2$ are uncorrelated, that is, 
    \begin{equation}
        \label{eq:PartialSumsUncorrelated}
        \ic |\xi (\A)|^2|\xi (\B)|^2\, dm = \left(\ic |\xi (\A)|^2\, dm\right) \left(\ic |\xi (\B)|^2\, dm\right)
    \end{equation}
    (see Theorem~\ref{thm:SquaredModuliPartialSums} in Section~\ref{sec:AleksandrovClarkMeasures}).
    The second property provides an exponential decay of the higher order correlations of the iterates.
    More concretely, let $\varepsilon_i = 1$ or $\varepsilon_i = -1$ for $i=1,2,\ldots, k$
    and $n_1 < \ldots < n_k$ be positive integers satisfying $n_j - n_{j-1} \geq q \geq 1$, $j=2, \ldots , k$.
    Denote $\boldsymbol{\varepsilon} = (\varepsilon_1,\ldots,\varepsilon_k)$
    and $\boldsymbol{n} = (n_1,\ldots,n_k).$
    For a positive integer $n$, denote by $f^{-n}$ the function defined by $f^{-n} (z) = \overline{f^n (z)}$, $z \in \circle$.
    It was shown in~\cite{ref:NicolauSolerGibert2022} that
    \begin{equation}
        \label{eq:HigherOrderCorrelations}
        \left|\ic \prod_{j=1}^k f^{\varepsilon_j n_j}\, dm\right|
        \leq C^k k! |f'(0)|^{\Phi(\boldsymbol{\varepsilon},\boldsymbol{n})}, \quad k=1,2, \ldots ,
    \end{equation}
    if $q$ is sufficiently large
    and where $\Phi$ is a certain function depending on the choice of indices
    that satisfies $\Phi(\boldsymbol{\varepsilon},\boldsymbol{n}) \geq kq/4$
    and that is well suited for summing over the indices afterwards
    (see Theorem~\ref{thm:HigherCorrelations} in Section~\ref{sec:AuxiliaryResults}).
    The main technical tool in the proof of both properties
    \eqref{eq:PartialSumsUncorrelated} and \eqref{eq:HigherOrderCorrelations} is
    the theory of  Aleksandrov-Clark measures and more concretely, the Aleksandrov Disintegration Theorem.

    We have tried to make this paper self-contained. We use some auxiliary results from~\cite{ref:NicolauSolerGibert2022} but, when possible, we have provided simple proofs. This paper is structured as follows.
    In Section~\ref{sec:HypothesisOptimality} we show that the growth condition~\eqref{eq:CoefficientGrowth} on the coefficients
    cannot be improved.
    Section~\ref{sec:AleksandrovClarkMeasures} contains a brief exposition on Aleksandrov-Clark measures and their application in the proof of Theorem~\ref{thm:InnerFunctionsCLT}.
    In Section~\ref{sec:AuxiliaryResults} we state some auxiliary results.
    Finally, we prove Theorem~\ref{thm:InnerFunctionsCLT} in Section~\ref{sec:ProofOfCLT}

    We thank the referee for reading carefully the paper and for several important suggestions. In particular the referee found several errors and inaccuracies which fortunately could be solved.

    \section{Optimality of the assumption on the coefficients}
    \label{sec:HypothesisOptimality}
    Here we discuss the optimal character of condition~\eqref{eq:CoefficientGrowth}.
    Observe that the results of Salem and Zygmund~\cite{ref:SalemZygmundLacunarySeries-I}
    can be applied in this context to inner functions of the form $f(z) = z^d$ with $d \geq 2.$
    Thus, since they showed that condition~\eqref{eq:CoefficientGrowth} is optimal for lacunary series,
    it cannot be improved for inner functions in general.
    Nonetheless, we expose their argument in slightly more general terms
    in order to show that it applies to other examples of inner functions $f.$

    Given a sequence of measurable functions $\{f_n\}$ defined on the unit circle,
    we say that they converge in distribution to a finite measure $\mu$ if
    \begin{equation*}
        \lim_{n \to \infty} \measset{z \in \circle\colon f_n (z) \in K} = \mu(K)
    \end{equation*}
    for every bounded set $K$ such that $\mu(\partial K) = 0.$
    This definition is usually stated in general for sequences of finite measures,
    but the current one will suffice for our purposes.
    Let us denote by $B(r)$ the closed ball of radius $r > 0$ centred at the origin.
    In the subsequent argument, it will be enough for us to restrict the general bounded sets $K$
    to balls $B(r),$ $r > 0.$
    Observe that in this case, if we define the functions $P_n(r) = \measset{z \in \circle\colon f_n(z) \in B(r)}$
    and $P(r) = \mu(B(r)),$ convergence in distribution implies that $\lim_{n\to\infty} P_n(r) = P(r)$
    whenever $r > 0$ is a point at which $P$ is continuous.
    
    \begin{proposition}
        \label{prop:OptimalityCoefficientsGrowth}
        Le $\{g_n\}$ be a sequence of measurable functions in the unit circle with $|g_n(z)| \leq 1$ for every $n \geq 1$ and for almost every $z \in \circle.$
    Let $\{a_n\}$ be a sequence of complex numbers that is not square summable and define
    \begin{equation*}
        S_N^2 = \sum_{n=1}^N |a_n|^2.
    \end{equation*}
    Denote
    \begin{equation*}
        G_N(z) = \sum_{n=1}^N a_n g_n (z),\qquad z \in \circle.
    \end{equation*}
    Assume that the sequence $\{G_N/S_N\}$ converges in distribution to a finite measure $\mu$
    and that it satisfies the uniform decay 
    \begin{equation}
        \label{eq:UniformDecay}
        \measset{z \in \circle\colon \left|\frac{G_N(z)}{S_N}\right| > r} \leq \varphi(r),
        \qquad \text{ for all } r>0,  N \geq 1,
    \end{equation}
    where $\varphi(r)$ is a positive decreasing function with $\lim_{r\to\infty} \varphi(r) = 0.$
    Assume that
        \begin{equation}
            \label{eq:CoefficientGrowthShort}
            \limsup_{N \to \infty}  \frac{|a_{N}|^2}{S_{N}^2} > 0.
        \end{equation}
        Then, $\mu$ is a probability measure and there exists $r_0 > 0$ such that
        \begin{equation}
            \label{eq:BoundednessInMeasure}
            \lim_{N\to\infty} \measset{z \in \circle\colon \left|\frac{G_N(z)}{S_N}\right| \leq r} = 1
        \end{equation}
        for all $r \geq r_0.$
        In addition $r_0 >0$ can be chosen so that $\mu(B(r_0)) = 1.$
    \end{proposition}
    \begin{proof}
    Given $r > 0$ and $N \geq 1,$ denote
    \begin{equation*}
        A_N(r) = \set{z \in \circle\colon \left|\frac{G_N(z)}{S_N}\right| \leq r}
    \end{equation*}
    and define the functions
    \begin{equation*}
        P_N(r) = \measure{A_N(r)}, \qquad
        P(r) = \mu(B(r)).
    \end{equation*}
    From the definitions, it is clear that $P_N$ and $P$ are nondecreasing functions.
    For later convenience, we recall that this implies that the set of points at which $P$ is discontinuous is countable.

    First we show that the uniform decay~\eqref{eq:UniformDecay} implies that $\mu$ is a probability measure.
    It is clear from the definition that for each $N \geq 1$ we have that $\lim_{r\to\infty} P_N(r) = 1.$
    Next, estimate~\eqref{eq:UniformDecay} for a given $r > 0$ is equivalent to
    \begin{equation*}
        P_N(r) = \measure{A_N(r)} \geq 1 - \varphi(r)
    \end{equation*}
    for all $N \geq 1.$
    In particular, if $r$ is a point at which $P$ is continuous, we have that $P(r) \geq 1 - \varphi(r).$
    Taking an increasing sequence $\{r_k\}$ of points of continuity of $P$ tending to infinity,
    these observations and the fact that $\lim_{N\to\infty} P_N(r_k) = P(r_k)$ for every $k \geq 1$
    imply that $\lim_{r\to\infty} P(r) = 1,$ as we wanted to see.

    Next we show~\eqref{eq:BoundednessInMeasure}.
    Recall that $P$ is nondecreasing and it has at most countably many discontinuities.
    We will show that there exists a point $r_0 > 0$ of continuity of $P$ such that
    $P(r) = 1$ for all $r \geq r_0.$
    Assume that this is not the case, so $P(r) < 1$ for any finite $r.$
    Because of~\eqref{eq:CoefficientGrowthShort}, there exists $\varepsilon > 0$ such that
    $|a_N|^2 > \varepsilon S_N^2$ for infinitely many values of $N.$
    For such values of $N$ we get that
    \begin{equation*}
        \frac{S_{N-1}^2}{S_N^2} = \frac{S_N^2 - |a_N|^2}{S_N^2} < 1 - \varepsilon.
    \end{equation*}
    Now write
    \begin{equation*}
        \frac{G_N}{S_N} = \frac{G_{N-1}}{S_{N-1}} \frac{S_{N-1}}{S_N} + \frac{a_N g_N}{S_N}.
    \end{equation*}
    For $r > 0$ and for such values of $N,$ observe that if $z \in A_{N-1}(r),$
    by the previous identity we have that $z \in A_N(r\sqrt{1-\varepsilon} + 1).$
    In particular, choosing $r > 0$ such that both $r$ and $r \sqrt{1-\varepsilon} + 1$ are continuity points of $P$
    and then taking the limit in $N$ (which exists by assumption),
    we find that for such values of $r$ it holds that $P(r) \leq P(r\sqrt{1-\varepsilon} + 1).$
    For $r \geq 1/(1-\sqrt{1-\varepsilon})$ it happens that $r\sqrt{1-\varepsilon} + 1 \leq r$ and, since $P$ is nondecreasing, we actually have that $P(r) = P(r\sqrt{1-\varepsilon} + 1).$

    \begin{figure}
        \centering
        \begin{tikzpicture}
                \begin{axis}[
                    xlabel = {$r$},
                    ylabel = {$P(r)$},
                    domain = 0:11,
                    ymin=0,
                    ymax=1.1,
                    smooth,
                    axis lines = left,
                    ytick={0.8,1},
                    yticklabels={$P(r_0)$,$1$},
                    xtick={2,4.3,9},
                    xticklabels={$r_0$,$r_1$,$r_2$}]

                    \addplot[color=black,
                             style=dashed]{1};
                    \addplot[color=black,
                             style=solid,
                             thick,
                             domain=2:4.3]{0.8};
                    \addplot[color=black,
                             style=dashed,
                             thick,
                             domain=4.3:9]{0.8};
                    \addplot[mark=*, only marks] coordinates {(2,0.8) (4.3,0.8)};
                    \draw[dashed] (2,0) -- (2,0.8);
                    \draw[dashed] (4.3,0) -- (4.3,0.8);
                    \draw[dashed] (9,0) -- (9,0.8);

                \end{axis}
        \end{tikzpicture}
        \caption{First, since $r_1 > 1/(1-\sqrt{1-\varepsilon})$ and $r_0 = r_1\sqrt{1-\varepsilon}+1,$
                 $P$ is constant and equal to $P(r_0)$ on the interval $[r_0, r_1].$
                 Then, the same argument shows that $P$ is constant on each interval $[r_n,r_{n+1}]$
                 for $n \geq 0.$}
        \label{fig:CumulativeDistributionArgument}
    \end{figure}

    Pick a point $r_0 > 0$ of continuity of $P$ such that $r_0 > 1/(1-\sqrt{1-\varepsilon})$ and define the sequence $\{r_n\}$ given by
    \begin{equation*}
        r_{n+1} = \frac{r_n-1}{\sqrt{1-\varepsilon}}
    \end{equation*}
    for $n \geq 1,$ which increases to infinity by our choice of $r_0.$
    Note that, since the $P$ has at most countably many discontinuities, one can choose $r_0$ as before
    in a way that $r_n$ is a point of continuity of $P$ for every $n \geq 0.$
    We use this sequence to prove that $P$ is constant and equal to $P(r_0) < 1$ on $[r_0,+\infty).$
    Since $r_1 > 1/(1-\sqrt{1-\varepsilon}),$ the previous argument shows that $P$ is constant
    on the interval $[r_1\sqrt{1-\varepsilon}+1,r_1] = [r_0,r_1],$
    so that $P(r) = P(r_0) < 1$ for all $r \in [r_0,r_1]$ (see Figure~\ref{fig:CumulativeDistributionArgument}).
    An inductive reasoning over the sequence $\{r_n\}$ gives that $P(r) = P(r_0) < 1$ for all $r \geq r_0.$
    However, this contradicts the fact that $\lim_{r\to\infty} P(r) = 1,$
    which must hold because we have shown that $\mu$ is a probability measure.
    \end{proof}

    We explain now how to use Proposition~\ref{prop:OptimalityCoefficientsGrowth} in the context of Theorem~\ref{thm:InnerFunctionsCLT}.
    Consider a fixed inner function $f$ with $f(0) = 0$ which is not a rotation and
    let $\{a_n\}$ be a sequence of complex numbers whose moduli are not square summable.
    Define $\sigma_N$ by~\eqref{eq:PartialSumVariance}.
    Recall that estimate~\eqref{eq:SNSigmaNComparability} asserts that there exists $C = C(|f'(0)|) \geq 1$
    such that $C^{-1} \leq \sigma_N/S_N \leq C$ (see also Theorem~\ref{thm:L2NormComparability}).
    However, the actual values of the sequence $\{\sigma_N/S_N\}$ depend both on $f$ and the particular sequence $\{a_n\}.$
    For this reason, further assume that for our fixed $f$ and $\{a_n\}$ the sequence $\{\sigma_N/S_N\}$ converges.

    Denote now
    \begin{equation*}
        G_N(z) = \sum_{n=1}^N a_n f^n(z),\qquad z \in \circle,
    \end{equation*}
    and express
    \begin{equation*}
        \frac{G_N(z)}{S_N} = \frac{\sigma_N}{\sqrt{2}S_N} \frac{\sqrt{2}G_N(z)}{\sigma_N}.
    \end{equation*}
    Therefore, under the assumption on $\{\sigma_N/S_N\},$
    if $\sqrt{2}G_N/\sigma_N$ converges in distribution to a standard complex normal variable,
    $G_N/S_N$ also converges to a (possibly nonstandard) normal variable.
    In particular, $G_N/S_N$ converges in distribution to a finite measure.
    Recall as well that $\sigma_N = \norm{G_N}{2},$ so that we have the uniform bound
    \begin{equation*}
        \norm{\frac{G_N}{S_N}}{2} \leq C.
    \end{equation*}
    Hence, by Chebyshev's inequality,
    the functions $\{G_N/S_N\}$ satisfy the uniform decay~\eqref{eq:UniformDecay} with $\varphi(r) = C^2/r^2.$
    Thus, if
    \begin{equation}
        \label{eq:NoCoefficientGrowth}
        \limsup_{N \to \infty}  \frac{|a_N|^2}{\sum_{n=1}^N |a_n|^2} >0,
    \end{equation}
    Proposition~\ref{prop:OptimalityCoefficientsGrowth} implies that $G_N/S_N$ should converge in distribution
    to a compactly supported measure, which is a contradiction.
    
    Observe that there is a class of inner functions $f$ for which $\{\sigma_N/S_N\}$ converges
    regardless of the sequence $\{a_n\}.$
    Namely, this happens for all inner functions $f$ satisfying $f'(0) = 0$
    in addition to the hypotheses of Theorem~\ref{thm:InnerFunctionsCLT}.
    In this case it actually holds that $\sigma_N = S_N$ for every $N \geq 1$ (see Theorem~\ref{thm:L2NormComparability})
    and the argument follows.
    This includes any function of the form $f(z) = z^d,$ with $d \geq 2,$ which induces a lacunary series.
    However, it also includes examples with no lacunarity, such as general finite Blaschke products $f$ with $f'(0) = 0.$

    \section{Aleksandrov-Clark measures}
    \label{sec:AleksandrovClarkMeasures}
    Given an analytic mapping $f$ from the unit disc into itself (not necessarily inner) and a point $\alpha \in \circle$,
    the function $(\alpha + f)/(\alpha - f)$ has positive real part
    and hence there exists a positive measure $\mu_\alpha = \mu_\alpha (f)$ in the unit circle
    and a constant $C_{\alpha} \in \R$ such that
    \begin{equation}
        \label{eq:ACMeasureHerglotz}
        \frac{\alpha + f(w)}{\alpha - f(w)} = \ic \frac{z + w}{z - w}\, d\mu_\alpha(z) + iC_{\alpha}, \quad w \in \D. 
    \end{equation}
    The measures $\{\mu_\alpha\colon \alpha \in \circle\}$ are called the Aleksandrov-Clark measures of the function $f$.
    Clark introduced them in his paper \cite{ref:ClarkOneDimensionalPerturbations}
    and many of their deepest properties were found by Aleksandrov in \cite{ref:AleksandrovMeasurablePartitionsCircle},
    \cite{ref:AleksandrovMultiplicityBoundaryValuesInnerFunctions} and \cite{ref:AleksandrovInnerFunctionsAndRelatedSpaces}.
    The two surveys \cite{ref:PoltoratskiSarasonACMeasures} and \cite{ref:SaksmanACMeasures}
    as well as \cite[Chapter~IX]{ref:CimaMathesonRossCauchyTransform} contain their main properties and a wide range of applications.
    Observe that if $f(0)=0$ then $\mu_\alpha$ are probability measures.
    Moreover, $f$ is inner if and only if $\mu_\alpha$ is a singular measure for some (all) $\alpha \in \circle$.
    From the definition it is clear that, if $f$ is an inner function, the mass of $\mu_\alpha$ is carried by the set $f^{-1}(\{\alpha\}) \subset \circle.$
    
    Assume $f(0)=0$.
    Computing the first two derivatives in formula \eqref{eq:ACMeasureHerglotz} and evaluating at the origin, we obtain 
    \begin{equation}
        \label{eq:FirstACMeasureMoment}
        \ic \overline{z}\, d\mu_{\alpha}(z) = f'(0) \overline{\alpha}, \quad \alpha \in \circle, 
    \end{equation}
    and
    \begin{equation*}
        \ic \overline{z}^2\, d\mu_{\alpha}(z) = \frac{f''(0)}{2} \overline{\alpha} + f'(0)^2 \overline{\alpha}^2, \quad \alpha \in \circle.
    \end{equation*}
    More generally, if we expand both terms of identity~\eqref{eq:ACMeasureHerglotz} in power series,
    we get that for every positive integer $l$ it holds that
    \begin{equation*}
        \ic {\overline{z}}^l\, d\mu_{\alpha}(z) 
        = \sum_{k=1}^l {\overline{\alpha}}^k  \ic f(z)^k {\overline{z}}^l\, dm(z), \quad \alpha \in \circle . 
    \end{equation*}
    Hence for any integer $l$,
    the $l$-th moment of $\mu_\alpha$ is a trigonometric polynomial in the variable $\alpha$ of degree less than or equal to $|l|$.
    Moreover, the coefficients of this polynomials are given by derivatives of powers of $f$ due to Cauchy's formula.

    Our main technical tool is the Aleksandrov Disintegration Theorem which asserts that 
    \begin{equation}
        \label{eq:AleksandrovDesintegrationThm}
        m = \ic \mu_\alpha\, dm(\alpha)
    \end{equation}
    holds true in the sense that
    \begin{equation*}
        \ic G\, dm = \ic \ic G(z)\, d\mu_\alpha(z)\, dm(\alpha)
    \end{equation*}
    for any integrable function $G$ on the unit circle.
    In other words, for any given analytic self-mapping $f$ of the unit disk,
    the Lebesgue measure is the average of the Aleksandrov-Clark measures $\{\mu_\alpha\}$ of $f.$
    
    Before showing the main application of the Aleksandrov Disintegration Theorem in our context,
    we mention a basic fact of inner functions, which is just a restatement of Löwner's Lemma.
    Recall that Löwner's Lemma claims that if $f$ is an inner function such that $f(0) = 0,$ then
    $\measure{f^{-1}(E)} = \measure{E}$ for any measurable set $E \subseteq \circle$
    (see for instance~\cite[Corollary~1.5]{ref:DoeringMane}).
    \begin{lemma}
        \label{lemma:DeconjugateInnerFunction}
        Let $f$ be an inner function with $f(0)=0$. 
        \begin{enumerate}[(a)]
            \item
            \label{lemma:item:DeconjugateInnerFunction}
            Let $G$ be an integrable function on $\circle$.
            Then
            \begin{equation*}
                \ic G(f(z))\, dm(z) = \ic G(z)\, dm(z).
            \end{equation*}
            
            \item
            \label{lemma:item:IteratesCovariance}
            Let $k<j$ be positive integers.
            Then 
            \begin{equation*}
                \ic \overline{f^k} f^j\, dm = f'(0)^{j-k}.
            \end{equation*}
        \end{enumerate}
    \end{lemma}
    \begin{proof}[Proof of Lemma~\ref{lemma:DeconjugateInnerFunction}]
        Assume that $G$ is the characteristic function of a measurable set $E \subset \circle$.
        Since $m(f^{-1} (E)) = m (E)$, the identity \eqref{lemma:item:DeconjugateInnerFunction} follows.
        The result for general integrable functions holds because of density of linear combinations of characteristic functions.
        Using \eqref{lemma:item:DeconjugateInnerFunction} and Cauchy formula, we have
        \begin{equation*}
            \ic \overline{f^k} f^j\, dm = \ic \overline{z} f^{j-k} (z)\, dm(z) = f'(0)^{j-k},
        \end{equation*}
        where the last equality follows from the chain rule.
    \end{proof}

    When proving Theorem~\ref{thm:InnerFunctionsCLT}, we will need to estimate integrals of products of iterates of an inner function $f.$
    We will use the Aleksandrov Disintegration Theorem to integrate factors of these products separately.
    The following lemma is a particular example of this procedure.
    However, the reader will easily see how to apply the same technique to more general products of iterates.
    The reason to choose this particular example is not only for clarity,
    but also because it will have a direct application in studying the $\Lp{2}$ norms of partial sums in Section~\ref{sec:AuxiliaryResults}.
    \begin{lemma}
        \label{lemma:IteratesProduct}
        Let $f$ be an inner function with $f(0)=0$.
        For $k =1,2,\ldots,p$, let $n_k, j_k$,  be positive integers such that 
        \begin{equation}
            \label{eq:SeparationHpthTwoIteratesProd}
            \max \{ n_k , j_k \} < \min \{n_{k+1} , j_{k+1} \} , \quad  k=1,\ldots,p-1 .
        \end{equation}
        Then 
        \begin{equation}
            \label{eq:IteratesProduct}
            \int_{\circle} \prod_{k=1}^p f^{n_k} \overline{f^{j_k}} dm = \prod_{k=1}^p \int_{\circle}  f^{n_k} \overline{f^{j_k}} dm . 
        \end{equation}
    \end{lemma}
    \begin{proof}[Proof of Lemma~\ref{lemma:IteratesProduct}]
        We argue by induction on $p$.
        Assume \eqref{eq:IteratesProduct} holds for $p-1$ products.
        We can assume $n_1 < j_1$.
        By part \eqref{lemma:item:DeconjugateInnerFunction} of Lemma \ref{lemma:DeconjugateInnerFunction} we have
        \begin{equation*}
            \ic \prod_{k=1}^p f^{n_k} \overline{f^{j_k}}\, dm
            = \ic z \overline{f^{j_1 - n_1} (z)} \prod_{k=2}^p f^{n_k - n_1} (z) \overline{f^{j_k - n_1} (z)}\, dm(z).  
        \end{equation*}
        Let $\{\mu_\alpha\colon \alpha \in \circle \}$ be the Aleksandrov-Clark measures of the inner function $f^{j_1 - n_1}$.
        The Aleksandrov Disintegration Theorem~\eqref{eq:AleksandrovDesintegrationThm}
        and the fact that the mass of $\mu_\alpha$ is carried by the set  $\set{z\in\circle\colon f^{j_1 - n_1}(z)=\alpha}$
        gives that last integral can be written as
        \begin{equation*}
            \ic \ic z \overline{\alpha} \prod_{k=2}^p f^{n_k - j_1}(\alpha) \overline{f^{j_k - j_1}(\alpha)}\, d\mu_\alpha(z)\, dm(\alpha). 
        \end{equation*}
        By~\eqref{eq:FirstACMeasureMoment} and part \eqref{lemma:item:IteratesCovariance} of Lemma \ref{lemma:DeconjugateInnerFunction},
        we have
        \begin{equation*}
            \ic z\, d\mu_\alpha(z) = {\overline{f'(0)}}^{j_1 - n_1} \alpha
            = \alpha \ic f^{n_1} \overline{f^{j_1}}\, dm . 
        \end{equation*}
        Hence
        \begin{equation*}
            \ic \prod_{k=1}^p f^{n_k} \overline{f^{j_k}}\, dm
            = \left(\ic f^{n_1} \overline{f^{j_1}}\, dm\right) \ic \prod_{k=2}^p f^{n_k - j_1} \overline{f^{j_k - j_1}}\, dm
        \end{equation*}
        and we can apply the inductive assumption.
        One more use of the invariance property of part \eqref{lemma:item:DeconjugateInnerFunction}
        of Lemma \ref{lemma:DeconjugateInnerFunction} finishes the proof.
    \end{proof}
    Linear combinations of iterates of inner functions are not independent,
    but under suitable conditions their correlations decay fast enough, which is sufficient for many applications. The following theorem states a simple condition under which
    the moduli squared of linear combinations of iterates of an inner function are uncorrelated.
    % The idea of its proof, which we do not reproduce here, is based on an iterative application of Lemma~\ref{lemma:IteratesProduct}
    % (see~\cite{ref:NicolauSolerGibert2022}).
    \begin{theorem}
        \label{thm:SquaredModuliPartialSums}
        Let $f$ be an inner function with $f(0)=0$.
        Let ${\A}_k$, $k =1,2,\ldots,p$, be finite collections of positive integers such that 
        \begin{equation}
            \label{eq:SeparationHpthPartialSums}
            \max \{ n\colon n \in {\A}_k \} < \min \{n\colon n \in {\A}_{k+1} \} , \quad k=1,\ldots , p-1.
        \end{equation}
        Consider 
        \begin{equation*}
            \xi_k = \sum_{n \in {\A}_k} a_n f^n.
        \end{equation*}
        Then 
        \begin{equation*}
            \ic \prod_{k=1}^p |\xi_k|^2\, dm = \prod_{k=1}^p \ic  |\xi_k|^2\, dm . 
        \end{equation*}
    \end{theorem}
    \begin{proof}
         At almost every point of the unit circle we have
        \begin{equation*}
            |\xi_k|^2
            = \sum_{n \in {\A}_k} |a_n|^2 + \sum( \overline{a_n} a_j \overline{f^n} f^j + \overline{a_j} a_n \overline{f^j} f^n ),
        \end{equation*}
        where the last sum is taken over all indices $n,j \in {\A}_k$ with $j>n$.
        Hence $\prod |\xi_k|^2$ can be written as a linear combination of terms of the form 
        \begin{equation*}
            \prod f^{n_k} \overline{f^{j_k}},
        \end{equation*}
        where $n_k, j_k \in {\A}_k$.
        Observe that \eqref{eq:SeparationHpthPartialSums} gives
        the assumption \eqref{eq:SeparationHpthTwoIteratesProd} in Lemma \ref{lemma:IteratesProduct}.
        Now Lemma \ref{lemma:IteratesProduct} finishes the proof.
    \end{proof}

    \section{Auxiliary results}
    \label{sec:AuxiliaryResults}
    This section is devoted to collect some results which will be used in the proof of  Theorem~\ref{thm:InnerFunctionsCLT}.
    
    Given an inner function $f$ on the unit disk and a positive integer $n,$
    we denote $f^{-n} = \overline{f^n}.$
    By applying Lemma~\ref{lemma:IteratesProduct} several times we can estimate expressions of the form
    \begin{equation*}
        \ic f^{\varepsilon_1 n_1} f^{\varepsilon_2 n_2} \dots f^{\varepsilon_k n_k}\, dm,
    \end{equation*}
    where $n_1,\ldots,n_k$ are positive integers and $\varepsilon_1,\ldots,\varepsilon_k \in \set{+1,-1}.$
    The following lemma states the corresponding results for some particular configurations with $4$ factors
    that will turn out to be useful later on.
    % (see~\cite{ref:NicolauSolerGibert2022} for a detailed proof).
    \begin{lemma}
        \label{lemma:FourFactors}
        Let $f$ be an inner function with $f(0)=0$ which is not a rotation.
        For $k =1,2,3,4$, let  $n_k$ be positive integers and $\varepsilon_k \in \set{+1,-1}$. Consider 
        \begin{equation*}
                I = I(\varepsilon_1 n_1 , \varepsilon_2 n_2 , \varepsilon_3 n_3 , \varepsilon_4 n_4)
                = \ic f^{\varepsilon_1 n_1} f^{\varepsilon_2 n_2} f^{\varepsilon_3 n_3} f^{\varepsilon_4 n_4}\, dm. 
            \end{equation*}
        \begin{enumerate}[(a)]
            \item
            \label{lemma:item:FourFactorsCancellation}
           Assume $\varepsilon_1 = - \varepsilon_2$, $ \varepsilon_3 = \varepsilon_4 $ and  $\max \{n_1 , n_2 \} < \min \{ n_3 , n_4 \}$. Then $I=0$. 
            
            \item
            \label{lemma:item:FourArbitraryFactors}
           Assume $n_1 < n_2 < n_3 < n_4$. If $\varepsilon_1 \varepsilon_2 = \varepsilon_3 \varepsilon_4 = -1$ then we have 
         $|I| = |f'(0)|^{n_2 - n_1 + n_4 - n_3}$. If $\varepsilon_1 =  \varepsilon_2$ and $n_3 - n_2 \geq 3$, then there exists a constant $C=C(f) >0$, independent of the indices $n_1, n_2, n_3, n_4$,
            such that $|I| \leq C |f'(0)|^{n_4 - n_1}$.

         \item
            \label{lemma:item:FirstFactorSquared}
            Let $n_1 < n_2 < n_3$ be positive integers with $n_2 - n_1 \geq 3$. Then there exists a constant $C=C(f) >0$, independent of the indices $n_1, n_2, n_3$,
            such that
            \begin{equation*}
               \left| \ic (f^{\varepsilon_1 n_1})^2  f^{\varepsilon_2 n_2} f^{\varepsilon_3 n_3}\, dm\right| \leq C |f'(0)|^{n_3 - n_1}.
            \end{equation*}

\item
            \label{lemma:item:nouFirstFactorSquared}
            Let $n_1 \leq n_2 < n_3$ be positive integers with $n_3 - n_2 \geq 3$. Then there exists a constant $C=C(f) >0$, independent of the indices $n_1, n_2, n_3$,
            such that
            \begin{equation*}
               \left| \ic f^{\varepsilon_1 n_1}  f^{\varepsilon_2 n_2} (f^{\varepsilon_3 n_3})^2\, dm\right| \leq C |f'(0)|^{n_3 - n_1}.
            \end{equation*}

             \end{enumerate}
    \end{lemma}

\begin{proof}[Proof of Lemma~\ref{lemma:FourFactors}]

We can assume $n_1 < n_2$ and $\varepsilon_1 = 1$. Since Lebesgue measure is invariant under $f$ we have
$$
I = \ic z f^{\varepsilon_2 (n_2 - n_1)}(z) f^{ \varepsilon_3 (n_3 - n_1)}(z) f^{\varepsilon_4 (n_4 - n_1)}(z) dm (z).
$$
Let $\{\mu_\alpha\colon \alpha \in \circle \}$ be the Aleksandrov-Clark measures of the inner function $f^{n_2 - n_1}$.
        The Aleksandrov Disintegration Theorem~\eqref{eq:AleksandrovDesintegrationThm}
        and the fact that the mass of $\mu_\alpha$ is carried by the set  $\set{z\in\circle\colon f^{n_2 - n_1}(z)=\alpha}$
        give that 
        \begin{equation*}
I =  \ic \ic z {\alpha}^{\varepsilon_2} f^{ \varepsilon_3 (n_3 - n_2)}(\alpha) f^{\varepsilon_4 (n_4 - n_2)}(\alpha) \, d\mu_\alpha(z)\, dm(\alpha). 
        \end{equation*}
        By~\eqref{eq:FirstACMeasureMoment} and part \eqref{lemma:item:IteratesCovariance} of Lemma \ref{lemma:DeconjugateInnerFunction},
        we have
        \begin{equation*}
            \ic z\, d\mu_\alpha(z) = {\overline{f'(0)}}^{n_2 - n_1} \alpha . 
        \end{equation*}
We deduce that
$$
I=  {\overline{f'(0)}}^{n_2 - n_1}  \ic  {\alpha}^{ 1 + \varepsilon_2} f^{ \varepsilon_3 (n_3 - n_2)}(\alpha) f^{\varepsilon_4 (n_4 - n_2)}(\alpha) \,  dm(\alpha). 
$$
If $\varepsilon_2 = -1$ and $\varepsilon_3 =  \varepsilon_4$ last integral vanishes and we obtain the statement in (a). If $\varepsilon_2 = -1$ and $\varepsilon_3 = -  \varepsilon_4$, the modulus of the last integral is $|f'(0)|^{n_4 - n_3}$ and we deduce the first part of (b). Assume now $\varepsilon_2 =1$. Then 
$$
I = {\overline{f'(0)}}^{n_2 - n_1}  \ic  {z}^{ 2} f^{ \varepsilon_3 (n_3 - n_2)}(z) f^{\varepsilon_4 (n_4 - n_2)}(z) \,  dm(z).
$$
Let $\{\sigma_\alpha\colon \alpha \in \circle \}$ be the Aleksandrov-Clark measures of the inner function $f^{n_3 - n_2}$. The Aleksandrov Disintegration Theorem~\eqref{eq:AleksandrovDesintegrationThm}
        and the fact that the mass of $\sigma_\alpha$ is carried by the set  $\set{z\in\circle\colon f^{n_3 - n_2}(z)=\alpha}$
        give that
$$
I =   {\overline{f'(0)}}^{n_2 - n_1} \ic \ic z^2 {\alpha}^{\varepsilon_3} f^{ \varepsilon_4 (n_4 - n_3)}(\alpha) \, d\sigma_\alpha (z)\, dm(\alpha). 
$$
By~\eqref{eq:FirstACMeasureMoment}
$$
\ic z^2 d\sigma_\alpha (z) = {\alpha}^2 {\overline{f'(0)}}^{n_3 - n_2} + \alpha b,
$$
and $|b| \leq C |f' (0)|^{n_3 - n_2}$ because $n_3 - n_2 \geq 3$. We deduce that
$$
|I| \leq C |f' (0)|^{n_3 - n_1} \max \left\{\left|\ic \alpha^{l} f^{\varepsilon_4 (n_4 - n_3)} (\alpha) \, dm (\alpha)\right| \right\},
$$
where the maximum is taken over all positive integers $l$ with $|l| \leq 3$. This gives the second part of statement (b). Similar arguments give parts (c) and (d).
    \end{proof}

    The next statement collects convenient estimates of $\Lp{2}$ and $\Lp{4}$ norms of linear combinations of iterates of an inner function (see~\cite{ref:NicolauSolerGibert2022} for details).
    Here, given $u,v\in\C,$ we denote by $\langle u,v \rangle = \Real(\overline{u}v)$
    the standard scalar product of $u$ and $v$ when considered as elements of $\R^2.$
    \begin{theorem}
        \label{thm:L2NormComparability}
        Let $f$ be an inner function with $f(0)=0$ which is not a rotation
        and let $\{a_n \}$ be a sequence of complex numbers with $\sum_n |a_n|^2 < \infty$.
        Consider 
        \begin{equation*}
            \xi = \sum_{n=1}^\infty a_n f^n
        \end{equation*}
        and 
        \begin{equation*}
            {\sigma}^2  = \sum_{n=1}^\infty |a_n|^2 + 2 \Real \sum_{k=1}^{\infty} f'(0)^k \sum_{n=1}^\infty \overline{a_n} a_{n+k} . 
        \end{equation*}
        
        \begin{enumerate}[(a)]
            \item
            \label{thm:item:L2NormEstimate}
            We have $\norm{\xi}{2}^2 ={\sigma}^2  $ and 
            \begin{equation*}
                \kappa^{-1} \sum_{n=1}^\infty |a_n|^2 \leq {\sigma}^2 \leq  \kappa \sum_{n=1}^\infty |a_n|^2 , 
            \end{equation*}
            where $\kappa= (1+|f'(0)|)(1- |f'(0)|)^{-1}$.
            
            \item
            \label{thm:item:ScalarProduct}
            For any $t \in \C$ we have
            \begin{equation*}
                \ic \langle t, \xi  \rangle^2\, dm = \frac{1}{2} |t|^2 {\sigma}^2. 
            \end{equation*}
            
            \item
            \label{thm:item:L4NormEstimate}
            There exists a constant $C=C(f) >0$ independent of the sequence $\{a_n\}$,
            such that $\norm{\xi}{4} \leq C \norm{\xi}{2}.$
        \end{enumerate}
    \end{theorem}

    In the proof of Theorem~\ref{thm:InnerFunctionsCLT} we will need to consider partial sums
    of the series of iterates of an inner function and also partial sums restricted to certain subsets of indices $\{\A(k)\}.$
    This auxiliary result, proved in~\cite{ref:NicolauSolerGibert2022}, will allow us to compare the $\Lp{2}$ norms of these partial sums.
    \begin{lemma}
        \label{lemma:PartialVariances}
        Let $\{a_n\}$ be a sequence of complex numbers and $\lambda \in \C$ with  $|\lambda| < 1.$
        Consider the sequence
        \begin{equation*}
            {\sigma}^2_N = \sum_{ n=1}^N |a_n|^2 + 2 \Real \sum_{k=1}^{N-1} {\lambda}^k \sum_{n=1}^{N-k}  \overline{a_n} a_{n+k},
            \quad  N=1,2 \ldots
        \end{equation*}
        For $N>1,$ let $\A(j) = \A(j,N),$ for $j=1,\ldots , M= M(N),$
        be pairwise disjoint sets of consecutive positive integers smaller than $N.$
        Consider
        \begin{equation*}
            {\sigma}^2 ({\A}(j)) =
            \sum_{ n \in {\A}(j)} |a_n|^2 + 2 \Real \sum_{k \geq 1} {\lambda}^k \sum_{n \in {\A}(j)\colon  n+k \in {\A}(j)} \overline{a_n} a_{n+k},
            \quad  j=1,2 \ldots , M. 
        \end{equation*}
        Let $\A = \cup {\A}(j).$
        Assume
        \begin{equation}
            \label{eq:PVFirst}
                \lim_{N \to \infty} \frac{\sum_{ n \in \A} |a_n|^2}{\sum_{ n=1}^N |a_n|^2} = 1
        \end{equation}
        and
        \begin{equation}
            \label{eq:PVSec}
            \lim_{j \to \infty} \sup_{N>1: j \leq M(N)} \frac{\max \{ |a_n|^2\colon n \in {\A}(j, N) \} } {\sum_{n \in {\A}(j, N) } |a_n|^2} = 0.
        \end{equation}
        Then 
        \begin{equation*}
            \lim_{N \to \infty} \frac{\sum_{j=1}^{M(N)} {\sigma}^2 ({\A}(j,N))}{{\sigma}^2_N} = 1.
        \end{equation*}
    \end{lemma}

    Another technical tool for the proof of Theorem~\ref{thm:InnerFunctionsCLT} is the following general version of Lemma~\ref{lemma:FourFactors} (see~\cite{ref:NicolauSolerGibert2022}). 
    It quantifies the correlation of iterates of an inner function when the number of iterations
    differ by a large amount.
    More concretely, if the minimum difference between the number of iterations on each iterate is at least $q,$ large enough,
    then we have certain control over correlations of up to $q$ different such iterates.
    Here we use the notation $f^{-n} = \overline{f^n}.$
    \begin{theorem}
        \label{thm:HigherCorrelations}
        Let $f$ be an inner function with $f(0)=0$ and $a = |f'(0)| < 1.$
        Let $1 \leq  k \leq q$ be positive integers.
        Let $\boldsymbol{\varepsilon} = \{\varepsilon_j\}_{j=1}^k$
        where $\varepsilon_j = 1$ or $\varepsilon_j = -1,$
        and let $\boldsymbol{n} = \{n_j\}_{j=1}^k$
        where $n_1 < n_2 < \ldots < n_k$ are positive integers with $n_{j+1} - n_j > q$ for any $j =1,2,\ldots , k-1.$ 
        Consider 
        \begin{equation*}
            I(\boldsymbol{\varepsilon},\boldsymbol{n}) = \left|\ \int_{\circle} \prod_{j=1}^k f^{\varepsilon_j n_j}\, dm \right| 
        \end{equation*}
        Then there exist constants $C= C(f)>0$,
        $q_0 = q_0 (f)>0$ independent of $\boldsymbol{\varepsilon}$ and of $\boldsymbol{n},$
        such that if $q \geq q_0$ we have    
        \begin{equation*}
            I(\boldsymbol{\varepsilon},\boldsymbol{n}) \leq C^k k! a^{ \Phi (\boldsymbol{\varepsilon},\boldsymbol{n}) }, k=1,2, \ldots,  
        \end{equation*}
        where $\Phi (\boldsymbol{\varepsilon},\boldsymbol{n}) = \sum_{j=1}^{k-1} \delta_j (n_{j+1} -n_{j}),$
        with $\delta_j \in \{0,1/2,1\}$ for any $j = 1, \ldots, k-1,$
        and with $\delta_1 = 1$ and $\delta_{k-1} \geq 1/2.$
        In addition, for $j = 2, \ldots, k-1$ the coefficient $\delta_j = 1$ if and only if $\delta_{j-1} = 0.$
        Furthermore, if $\delta_{j-1} > 0,$ the coefficient $\delta_j$
        depends on $\varepsilon_{j+1}, \ldots, \varepsilon_k$ and $n_j, \ldots, n_k$ for $j = 2, \ldots, k-1.$ 
    \end{theorem}

    The last auxiliary result of this section will be used to simplify the proof of Theorem~\ref{thm:InnerFunctionsCLT}.
    Its proof is an elementary application of Cauchy-Schwarz's inequality (see~\cite{ref:NicolauSolerGibert2022}).
    \begin{lemma}
        \label{lemma:IntegralLimit}
        Let $\{f_n \}$, $\{g_n \}$ be two sequences of measurable functions defined at almost every point of the unit circle.
        Assume that there exists a constant $C>0$ such that the following conditions hold
        \begin{enumerate}[(a)]
            \item
            $\sup_n \|f_n \|_2 \leq C $ and 
            \begin{equation*}
                \lim_{n \to \infty} \ic f_n\, dm =1
            \end{equation*}
            
            \item
            $ g_n (z) > -C$ for almost every $z \in \circle $ and $\lim_{n \to \infty} \|g_n \|_2 =0.$ 
        \end{enumerate}
        Then 
        \begin{equation*}
            \lim_{n \to \infty} \ic f_n e^{-g_n}\, dm = 1. 
        \end{equation*}
    \end{lemma}

    \section{Proof of the Central Limit Theorem}
    \label{sec:ProofOfCLT}
    The proof of Theorem~\ref{thm:InnerFunctionsCLT} relies on splitting the partial sum
    \begin{equation*}
        \sum_{n=1}^N a_n f^n = \sum_k (\xi_k + \eta_k )
    \end{equation*}
    into certain blocks of consecutive terms which depend on $N$ and that are of two alternating types: $\xi_k$ and $\eta_k$.
    In other words, we will have a block $\eta_k$ between blocks $\xi_k$ and $\xi_{k+1}.$ The main idea in this construction relies on balancing two counteracting properties. On the one hand $\|\sum \eta_k \|_2 / \sigma_N$ must be small so that $\sum \eta_k$ becomes irrelevant. On the other hand the number of terms of each block $\eta_k$ must be large so that the correlations between different blocks $\xi_k$ decay sufficiently fast.
     % This is done in a way such that most of the $\Lp{2}$ norm of the initial partial sum is concentrated
    % on the blocks $\xi_k,$ so that the blocks $\eta_k$ turn out to be irrelevant as a whole.
    % At the same time though, it is ensured that the number of terms in each block $\eta_k$ is sufficiently large
    % so that the correlations between consecutive blocks $\xi_k$ are very small and we can treat them as independent up to a small error.
    %Moreover, the ratio between the total $\Lp{2}$ mass of the blocks $\eta_k$ and the $\Lp{2}$ norm of the partial sum at step $N$
    % will decay to $0$ as $N \to \infty,$
    % while at the same time the blocks $\eta_k$ will have a minimum number of terms that will increase with $N$ fast enough
    % for the error due to the lack of independence to decay as well as $N \to \infty.$
    This procedure was used by M.~Weiss in order to prove the Law of Iterated Logarithm for lacunary series~\cite{ref:WeissLILLacunarySeries}
    and is detailed in our context in the next lemma. Let $|\A|$ denote the number of elements of a set $\A$ of integers.
    \begin{lemma}
        \label{lemma:BlockConstruction}
        Consider an inner function $f$ with $f(0) = 0$ which is not a rotation and a sequence $\{a_n\}$ of complex numbers.
        Denote
        \begin{equation*}
            \sigma_N^2
            = \sum_{n=1}^{N} |a_n|^2 + 2 \Real \sum_{k=1}^{N-1} f'(0)^k \sum_{n=1}^{N-k} \overline{a_n} a_{n+k},
        \end{equation*}
        \begin{equation*}
            S_N^2 = \sum_{n=1}^N |a_n|^2
        \end{equation*}
        and assume that there exists a nonincreasing function $\varphi$ with $\lim_{N\to\infty} \varphi(N) = 0$ such that
        \begin{equation}
            \label{eq:QuantitativeGrowthCondition}
            \sup \{|a_n|^2\colon 1 \leq n \leq N\} \leq \varphi(N) S_N^2.
        \end{equation}
        If $N$ is large enough, then we can choose indices $0 \leq M_k < N_{k+1} < M_{k+1} \leq N$ for $0 \leq k \leq Q_N-1,$
        with the possible exception that $N_{Q_N} = M_{Q_N} = N,$
        such that if we define the blocks of consecutive integers
        \begin{equation}
            \label{eq:IndexBlockDefinition}
            \A(k) = \{n \in \N\colon M_{k-1} < n \leq N_k\}, \quad \B(k) = \{n \in \N\colon N_k < n \leq M_k\}, \quad 1 \leq k \leq Q_N
        \end{equation}
        and the block sums
        \begin{equation}
            \label{eq:BlockSumsDefinition}
            \xi_k = \sum_{n\in\A(k)} a_n f^n , \quad \eta_k = \sum_{n \in \B(k)} a_n f^n , \quad 1 \leq k \leq Q_N, 
        \end{equation}
        it holds that,
        \begin{equation}
            \label{eq:NumberOfLargeBlocks}
            \lim_{N\to\infty} Q_N\varphi(N)^{1/8} = \frac{1}{2},
        \end{equation}
        \begin{equation}
            \label{eq:NormOfLargeBlock}
         \varphi(N)^{1/8} \sigma_N^2 \lesssim   \norm{\xi_k}{2}^2 \lesssim \varphi(N)^{1/8} \sigma_N^2,
        \end{equation}
        %\begin{equation}
        %    \label{eq:TailBlockBound}
         %   \norm{\sum_{n=1}^N a_n f^n - \sum_{k=1}^{Q_N} (\xi_k + \eta_k )}{2}^2 \leq  3 C(f) \% varphi(N)^{1/8} \sigma_N^2, 
        % \end{equation}
        \begin{equation}
            \label{eq:SumOfShortBlocks}
           \frac{1}{{\sigma}_N} \norm{\sum_{n=1}^N a_n f^n - \sum_{k=1}^{Q_N} \xi_k}{2} \lesssim \varphi (N) ^{1/16}
        \end{equation}
        and
        \begin{equation}
            \label{eq:NumberOfTermsPerBlock}
            |\A(k)| \geq \varphi(N)^{-7/8}, \quad |\B(k)| \geq \frac{1}{2} \varphi(N)^{-1/2}, \quad  k=1,2,\ldots, Q_N,
        \end{equation}
        with the possible exception that $\B(Q_N)$ may be empty.
    \end{lemma}
    
    \noindent In particular, observe that~\eqref{eq:NumberOfLargeBlocks} and~\eqref{eq:NumberOfTermsPerBlock} imply that
    \begin{equation*}
        Q_N \leq \varphi(N)^{-1/8} \leq \varphi(N)^{-1/2} \leq |\B(k)|, \quad 1 \leq k \leq Q_N-1, 
    \end{equation*}
     if $N$ is sufficiently large. Hence the number of blocks $\xi_k$ is much smaller than the number of (possibly null) terms in the blocks $\eta_k$. 
   
    \begin{proof}
        We may assume that $N$ is large enough for $\varphi(N) < 1$ to be already small. Recall that, by Theorem~\ref{thm:L2NormComparability}, we have that $S_N^2 \simeq \sigma_N^2.$ 
        First we define an auxiliary sequence of indices $0 \leq J_k \leq J_{k+1} \leq N$ as follows.
        Pick $J_0 = 0$ and let $J_1$ be the smallest positive integer such that
        \begin{equation*}
            \sum_{n=1}^{J_1} |a_n|^2 \geq \varphi(N)^{1/8} S_N^2. 
        \end{equation*}
        Now define the auxiliary block of consecutive integers
        \begin{equation*}
            \J(1) = \{n\in\N\colon J_0 < n \leq J_1\}.
        \end{equation*}
        It is clear that the minimality of $J_1$ and the bound~\eqref{eq:QuantitativeGrowthCondition}
        on the coefficients imply that
        \begin{equation*}
            \sum_{n\in\J(1)} |a_n|^2 \leq \varphi(N)^{1/8} S_N^2 + |a_{J_1}|^2 \leq (\varphi(N)^{1/8} + \varphi(N)) S_N^2. 
        \end{equation*}
        Assume that we have chosen $J_0,J_1,\ldots,J_{k-1}$ and
        defined the blocks of consecutive integers $\J(1),\ldots,\J(k-1).$
        Then pick $J_k$ to be the smallest positive integer such that
        \begin{equation*}
            \sum_{n=J_{k-1}+1}^{J_k} |a_n|^2 \geq \varphi(N)^{1/8} S_N^2, 
        \end{equation*}
        and let
        \begin{equation*}
            \J(k) = \{n\in\N\colon J_{k-1} < n \leq J_k\}.
        \end{equation*}
        As before, the minimality of $J_k$ and~\eqref{eq:QuantitativeGrowthCondition} give that
        \begin{equation}
            \label{eq:AuxiliaryBlockBound}
            \varphi(N)^{1/8} S_N^2 \leq \sum_{n\in\J(k)} |a_n|^2 \leq (\varphi(N)^{1/8} + \varphi(N)) S_N^2.
        \end{equation}
        We continue this process until we reach $J_{P_N}$ such that
        \begin{equation}
            \label{eq:ResidueBound}
            \sum_{n=J_{P_N}+1}^{N} |a_n|^2 < \varphi(N)^{1/8} S_N^2
        \end{equation}
        and denote $T_N^2 = \sum_{n=J_{P_N}+1}^{N} |a_n|^2$ (see Figure~\ref{fig:JkBlocks}).
        Observe that by estimate~\eqref{eq:AuxiliaryBlockBound}, summing over $k,$ we have that 
        \begin{equation*}
            \varphi(N)^{1/8} S_N^2 P_N \leq S_N^2 - T_N^2 \leq (\varphi(N)^{1/8} + \varphi(N)) S_N^2 P_N.
        \end{equation*}
        Thus, since $\varphi(N) \to 0,$ by~\eqref{eq:ResidueBound} we get that the number $P_N$ of auxiliary blocks satisfies
        \begin{equation}
            \label{eq:NumberOfAuxiliaryBlocks}
            \lim_{N\to\infty} \varphi(N)^{1/8} P_N = 1.
        \end{equation}
        Also note that, since $|a_n|^2 \leq \varphi(N) S_N^2$ for $n \leq N,$
        we deduce from the lower bound in~\eqref{eq:AuxiliaryBlockBound} that $\varphi(N) S_N^2 |\J(k)| \geq \varphi(N)^{1/8} S_N^2,$ $1 \leq k \leq P_N$. 
        % \begin{equation*}
        %    \varphi(N) S_N^2 |\J(k)| \geq \varphi(N)^{1/8} S_N^2.
        % \end{equation*}
        Hence, the number of indices in $\J(k)$ satisfies 
        \begin{equation}
            \label{eq:NumberOfTermsAuxiliaryBlock}
            |\J(k)|  \geq \varphi(N)^{-7/8}, \quad  1 \leq k \leq P_N.
        \end{equation}

        \begin{figure}
            \centering
            \begin{tikzpicture}
                \draw (0,0) rectangle ++(1,1);
                \node at (0.5,0.5) {$1$};
                \draw (1,0) rectangle ++(1,1);
                \node at (1.5,0.5) {$2$};
                \draw (2,0) rectangle ++(1,1);
                \node at (2.5,0.5) {$3$};
                \draw (3,0) rectangle ++(1,1);
                \node at (3.5,0.5) {$\dots$};
                \draw (4,0) rectangle ++(1,1);
                \node at (4.5,0.5) {$J_1$};
                \draw (5,0) rectangle ++(1,1);
                \node at (5.5,0.5) {$\dots$};
                \draw (6,0) rectangle ++(2,1);
                \node at (7,0.5) {$J_{P_N-1}+1$};
                \draw (8,0) rectangle ++(1,1);
                \node at (8.5,0.5) {$\dots$};
                \draw (9,0) rectangle ++(2,1);
                \node at (10,0.5) {$J_{P_N}$};
                \draw (11,0) rectangle ++(2,1);
                \node at (12,0.5) {$J_{P_N}+1$};
                \draw (13,0) rectangle ++(1,1);
                \node at (13.5,0.5) {$\dots$};
                \draw (14,0) rectangle ++(1,1);
                \node at (14.5,0.5) {$N$};
                \draw (0,1) rectangle ++(5,1);
                \node at (2.5,1.5) {$\mathcal{J}(1)$};
                \draw (5,1) rectangle ++(1,1);
                \node at (5.5,1.5) {$\dots$};
                \draw (6,1) rectangle ++(5,1);
                \node at (8.5,1.5) {$\mathcal{J}(P_N)$};
                \draw (11,1) rectangle ++(4,1);
                \node at (13,1.5) {coefficients in $T_N$};
            \end{tikzpicture}
            \caption{The blocks of indices $\J(k)$ are pairwise disjoint blocks of consecutive indices
                     that contain all positive integers up to $J_{P_N}.$}
            \label{fig:JkBlocks}
        \end{figure}

        Next, we modify the auxiliary blocks $\J(k)$ to obtain the blocks $\A(k)$ and $\B(k)$
        (see Figure~\ref{fig:AkBkBlocks}).
        Fix $1 \leq k \leq \lfloor P_N/2 \rfloor.$
        By \eqref{eq:NumberOfTermsAuxiliaryBlock},
        each block $\J(2k)$ has length larger than $\varphi(N)^{-7/8},$
        which is larger than $\varphi(N)^{-1/2}.$
        We pick $p_k$ smaller pairwise disjoint blocks of consecutive integers $\B(k,j)$ of lengths $ \lfloor \varphi(N)^{-1/2} \rfloor$. Observe that the blocks $\B(k,j)$ may not exhaust $J(2k)$. 
%        \begin{equation*}
 %           \lfloor \varphi(N)^{-1/2} \rfloor \leq |\B(k,j)| \leq \lfloor % \varphi(N)^{-1/2} \rfloor + 1.
 %       \end{equation*}
        Note that the estimate~\eqref{eq:NumberOfTermsAuxiliaryBlock} on the number of indices in $\J(2k),$ gives that we can get at least
        \begin{equation}
            \label{eq:NumberOfSmallBlocks}
            p_k \geq \frac{|\J(2k)|}{\varphi(N)^{-1/2}} \geq  \varphi(N)^{-3/8}
        \end{equation}
        such shorter blocks.
        Now we pick $\B(k)$ to be one of the blocks $\B(k,j)$ such that the sum
        \begin{equation*}
            \sum_{n\in\B(k,j)} |a_n|^2
        \end{equation*}
        is minimal.
        It turns out that
        \begin{equation}
            \label{eq:l2SumShortBlocks}
            \sum_{n\in\B(k)} |a_n|^2 \leq \frac{1}{p_k} \sum_{n\in\J(2k)} |a_n|^2 \leq 2 \varphi(N)^{1/2} S_N^2,
        \end{equation}
        where the last inequality follows from~\eqref{eq:AuxiliaryBlockBound} and~\eqref{eq:NumberOfSmallBlocks}.

        \begin{figure}
            \centering
            \begin{tikzpicture}
                \draw (0,1) rectangle ++(3,1);
                \node at (1.5,1.5) {$\mathcal{J}(1)$};
                \draw (3,1) rectangle ++(1,1);
                \node at (3.5,1.5) {$\mathcal{B}(\!1,\!1\!)$};
                \draw (4,1) rectangle ++(1,1);
                \node at (4.5,1.5) {$\mathcal{B}(\!1,\!2\!)$};
                \draw (5,1) rectangle ++(1,1);
                \node at (5.5,1.5) {$\mathcal{B}(\!1,\!3\!)$};
                \draw (6,1) rectangle ++(1,1);
                \node at (6.5,1.5) {$\mathcal{B}(\!1,\!4\!)$};
                \draw (7,1) rectangle ++(3,1);
                \node at (8.5,1.5) {$\mathcal{J}(3)$};
                \draw (10,1) rectangle ++(1,1);
                \node at (10.5,1.5) {$\mathcal{B}(\!2,\!1\!)$};
                \draw (11,1) rectangle ++(1,1);
                \node at (11.5,1.5) {$\mathcal{B}(\!2,\!2\!)$};
                \draw (12,1) rectangle ++(1,1);
                \node at (12.5,1.5) {$\mathcal{B}(\!2,\!3\!)$};
                \draw (13,1) rectangle ++(2,1);
                \node at (14,1.5) {$\mathcal{J}(5)$};

                \draw (0,0) rectangle ++(5,1);
                \node at (2.5,0.5) {$\mathcal{A}(1)$};
                \draw (5,0) rectangle ++(1,1);
                \node at (5.5,0.5) {$\mathcal{B}(1)$};
                \draw (6,0) rectangle ++(6,1);
                \node at (9,0.5) {$\mathcal{A}(2)$};
                \draw (12,0) rectangle ++(1,1);
                \node at (12.5,0.5) {$\mathcal{B}(2)$};
                \draw (13,0) rectangle ++(2,1);
                \node at (14,0.5) {$\mathcal{A}(3)$};
            \end{tikzpicture}
            \caption{In each block $\J(2k)$ consider smaller blocks $\B(k,j).$
                     Then, we choose $\B(k)$ to be one of these blocks
                     for which $\sum_{n\in\B(k,j)} |a_n|^2$ is minimal.
                     Finally, the blocks $\A(k)$ are the sets of indices between blocks $\B(k-1)$ and $\B(k).$}
            \label{fig:AkBkBlocks}
        \end{figure}

        We define now the sequences $M_k$ and $N_k.$
        For clarity, we will assume that $P_N$ is even, since the minor modifications when it is odd will be obvious.
        Set $Q_N = P_N /2$, $M_0 = 0$ and, $N_k = \min \B(k) - 1$, $M_k = \max \B(k) $ for $1 \leq k \leq Q_N,$ and 
        \begin{equation*}
            \A(k) = \{n\in\N\colon M_{k-1} < n \leq N_k\}, \qquad 1 \leq k \leq Q_N.
        \end{equation*}
        Note that~\eqref{eq:IndexBlockDefinition} is clear by the definition of $M_k$ and $N_k.$
        In addition, by the construction of $\B(k)$ for $1 \leq k \leq Q_N$ we have that
        \begin{equation*}
            |\B(k)| \geq \frac{1}{2} \varphi(N)^{-1/2} . 
        \end{equation*}
        Also, taking into account that each block $\A(k),$ $1 \leq k \leq Q_N$, contains some block $\J(i)$
        with $1 \leq i \leq P_N$ odd, estimate~\eqref{eq:NumberOfTermsAuxiliaryBlock} gives that 
        \begin{equation*}
            |\A(k)| \geq \varphi(N)^{-7/8}
        \end{equation*}
        and this proves~\eqref{eq:NumberOfTermsPerBlock}.
        After defining the sets of indices $\A(k)$ and $\B(k),$ the block sums $\xi_k$ and $\eta_k$
        are given by~\eqref{eq:BlockSumsDefinition} for $1 \leq k \leq Q_N.$
        Also note that for any $1 \leq k \leq Q_N$ we have $\A(k) \subset \J(2k-2) \cup \J(2k-1) \cup \J(2k)$ (where we consider $\J(0) = \emptyset$). Hence applying Theorem~\ref{thm:L2NormComparability} to $\xi_k$ we find
        \begin{equation*}
            \norm{\xi_k}{2}^2 \lesssim \sum_{n\in\J(i-1) \cup \J(i) \cup \J(i+1)} |a_n|^2 \lesssim \varphi(N)^{1/8} \sigma_N^2, \quad 1 \leq k \leq Q_N, 
        \end{equation*}
        which is property~\eqref{eq:NormOfLargeBlock}.
        
        When $P_N$ is even, the construction already gives that $Q_N = P_N/2,$ while 
        if $P_N$ is odd, then $Q_N = (P_N+1)/2.$
        In any of these two cases, using~\eqref{eq:NumberOfAuxiliaryBlocks} we get~\eqref{eq:NumberOfLargeBlocks}.
        The other differences that we would have when $P_N$ is odd are that we would have one more block $\A(k)$ than $\B(k)$
        and that the last index in the sequences $\{M_k\}$ and $\{N_k\}$
        would be $N_{Q_N}$ instead of $M_{Q_N}.$

       We are only left with checking~\eqref{eq:SumOfShortBlocks}. First denote the sums over indices
        that are not in any of the blocks $\xi_k$ and $\eta_k$, by
        \begin{equation*}
            R_N^2 = \sum_{n=M_{Q_N}+1}^N |a_n|^2, \qquad \rho_N = \sum_{n=M_{Q_N}+1}^N a_n f^n.
        \end{equation*}
        In particular, note that
        \begin{equation*}
            R_N^2 = \sum_{n=M_{Q_N}+1}^{J_{P_N}} |a_n|^2 + \sum_{n=J_{P_N}+1}^{N} |a_n|^2 \leq (2\varphi(N)^{1/8}+\varphi(N)) S_N^2.
        \end{equation*}
        Thus, using Theorem~\ref{thm:L2NormComparability} twice we see that
        \begin{equation}
            \label{eq:TailNorm}
            \norm{\rho_N}{2}^2 = \norm{\sum_{n=1}^N a_n f^n - \sum_{k=1}^{Q_N} (\xi_k + \eta_k )}{2}^2 \lesssim R_N^2 \lesssim  \varphi(N)^{1/8} \sigma_N^2.
        \end{equation}
%        which is precisely~\eqref{eq:TailBlockBound}.
        Note that
        \begin{equation*}
            \sum_{n=1}^N a_n f^n - \sum_{k=1}^{Q_N} \xi_k = \sum_{k=1}^{Q_N} \eta_k + \rho_N.
        \end{equation*}
        Then using~\eqref{eq:l2SumShortBlocks},~\eqref{eq:TailNorm} and~\eqref{eq:NumberOfLargeBlocks}, we deduce
        \begin{equation}
            \label{eq:2TailBoundExpanded}
            \begin{split}
                \norm{\sum_{n=1}^N a_n f^n - \sum_{k=1}^{Q_N} \xi_k}{2}  \leq & \sum_{k=1}^{Q_N} \|\eta_k\|_2 + \|\rho_N \|_2 \lesssim \\
                & \lesssim Q_N \varphi (N)^{1/4} \sigma_N  + \varphi (N)^{1/16} \sigma_N \lesssim \varphi (N)^{1/16} \sigma_N. 
            \end{split}
        \end{equation}
        This concludes the proof of~\eqref{eq:SumOfShortBlocks}.
        \end{proof}

    We use Lemma~\ref{lemma:BlockConstruction} to prove Theorem~\ref{thm:InnerFunctionsCLT}.
    The idea is that we will discard the blocks $\eta_k$ and the tail $\rho_N$ as their total $\Lp{2}$ norm is irrelevant
    compared to that of the full partial sum.
    Then, the correlation between the blocks $\xi_k$ will be controlled due to the length of the blocks $\eta_k$
    that separate them.
    In practice, this will imply that the blocks $\xi_k$ will behave almost as if they were independent.

    \begin{proof}[Proof of Theorem~\ref{thm:InnerFunctionsCLT}]
        For a given value of $N$ split the partial sum
        \begin{equation*}
            \sum_{n=1}^N a_n f^n
        \end{equation*}
        into the blocks $\xi_k$ and $\eta_k,$ with $1 \leq k \leq Q_N,$ given by Lemma~\ref{lemma:BlockConstruction}.
        Observe that, because of property~\eqref{eq:SumOfShortBlocks} and Chebyshev's inequality, we only need to see that
        \begin{equation*}
            T_N = \frac{\sqrt{2}}{\sigma_N} \sum_{k=1}^{Q_N} \xi_k
        \end{equation*}
        tends in distribution to a standard complex normal random variable as $N$ tends to infinity.
        We split the proof into several parts.

        \textit{1. The characteristic function of $T_N.$}
        By the Levi Continuity Theorem, it is sufficient to show that for any complex number $t$ we have 
        \begin{equation}
            \label{eq:FourierTransform}
            \varphi_N (t) = \ic e^{i \langle t , T_N \rangle }\, dm  \to e^{-|t|^2 / 2} ,
            \quad \text{ as } \quad N \to \infty.
        \end{equation}
        As before $\langle t,w \rangle = \Real (\overline{t} w)$ is the standard scalar product in the plane.
        Fix $t \in \C$. Our first step is to show the approximation 
        \begin{equation}
            \label{eq:FourierTransformApprox}
            \lim_{N \to \infty}  \left(\varphi_N (t) - \ic \prod_{k=1}^{Q_N} \left(1+ \frac{ i \sqrt{2}\langle t , \xi_k \rangle}{\sigma_N} \right)
            \exp{\left( - \frac{{\langle t , \xi_k \rangle}^2}{{\sigma_N}^2} \right)}\,  dm\right)   = 0.
        \end{equation}
        We start with some apriori estimates. Given $\delta >0$,
        consider the sets $E_k = E_k (\delta, N) =  \{z \in \circle\colon |\xi_k (z)| > \delta \sigma_N\}$, $k =1,2,\ldots , Q_N$.
        By part \eqref{thm:item:L4NormEstimate} of Theorem~\ref{thm:L2NormComparability} and estimate~\eqref{eq:NormOfLargeBlock}
        we have $\norm{\xi_k}{4}^4 \lesssim  \varphi(N)^{1/4} \sigma_N^4.$
        Chebyshev's inequality and estimate~\eqref{eq:NumberOfLargeBlocks} give 
        \begin{equation*}
            \sum_{k=1}^{Q_N} m(E_k) \lesssim \frac{\varphi(N)^{1/4} \sigma_N^4 Q_N}{{\delta}^4 \sigma_N^4} \lesssim \frac{\varphi(N)^{1/8}}{{\delta}^4}
        \end{equation*}
        if $N$ is sufficiently large.
        Given $\mu > 1$, consider the set
        \begin{equation*}
            E_0 = E_0 (\mu, t, N) = \set{z \in \circle\colon \sum_{k=1}^{Q_N} {\langle t , \xi_k (z) \rangle}^2 > \mu \sigma_N^2}.  
        \end{equation*}
        Part \eqref{thm:item:ScalarProduct} of Theorem~\ref{thm:L2NormComparability}, estimates~\eqref{eq:NumberOfLargeBlocks}, ~\eqref{eq:NormOfLargeBlock} and
        Chebyshev's inequality yield 
        \begin{equation*}
            m(E_0) \lesssim \frac{|t|^2 \varphi(N)^{1/8} \sigma_N^2 Q_N}{\mu \sigma_N^2} \lesssim  \frac{|t|^2}{\mu }, 
        \end{equation*}
        if $N$ is sufficiently large.
        Hence the set 
        \begin{equation*}
            E= E(\delta, \mu , t, N) =  \bigcup_{k=0}^{Q_N} E_k
        \end{equation*}
        satisfies
        \begin{equation}
        \label{E}
            m(E) \lesssim \left(\frac{\varphi(N)^{1/8}}{{\delta}^4} + \frac{|t|^2}{\mu}\right). 
        \end{equation}
        Now we will choose appropriate constants $\delta >0$ and $\mu >1$. Using the elementary identity 
        \begin{equation*}
            \exp{(z)} = (1+z) \exp \left(\frac{z^2}{2} + o(|z|^2)\right), 
        \end{equation*}
        where $o(|z|^2)/ |z|^2 \to 0$ as $z \to 0$, we deduce 
        \begin{equation*}
            \exp{( i \langle t , T_N \rangle )}
            = \left(\prod_{k=1}^{Q_N} \left(1+ \frac{ i \sqrt{2} \langle t , \xi_k \rangle}{\sigma_N} \right)
              \exp{\left(-  \frac{{\langle t , \xi_k \rangle}^2}{\sigma_N^2}\right) } \right)
              \exp{ \left( \sum_{k=1}^{Q_N} o \left(\frac{{ \langle t , \xi_k \rangle}^2}{\sigma_N^2} \right)\right)}.
        \end{equation*}
        Fix $\varepsilon >0$ and let $\mu >1$ be a constant to be fixed later. Note that $ {\langle t , \xi_k (z) \rangle}^2 / \sigma_N^2 \leq \delta^2 |t|^2$ if $z \in \circle \setminus E_k$, $k=1, \ldots, Q_N$. So picking $\delta = \delta (\varepsilon , t) >0$ sufficiently small we have
        \begin{equation*}
        \left|\Real\sum_{k=1}^{Q_N} o\left(\frac{{ \langle t , \xi_k (z) \rangle}^2}{\sigma_N^2} \right)\right|
            \leq \varepsilon \sum_{k=1}^{Q_N} \frac{{ \langle t , \xi_k (z) \rangle}^2}{\sigma_N^2} \leq    \varepsilon \mu , \quad z \in \circle \setminus E.
        \end{equation*}
        Note that the last inequality holds because $z \in \circle \setminus E_0.$
        Hence
        \begin{equation*}
            \begin{split}
                &\left|\int_{\circle \setminus E} \exp{(i \langle t , T_N \rangle ) }\, dm
                 - \int_{\circle \setminus E} \prod_{k=1}^{Q_N} \left(1+ \frac{ i \sqrt{2} \langle t , \xi_k \rangle}{\sigma_N}\right)
                   \exp{\left(- \frac{{\langle t , \xi_k \rangle}^2}{\sigma_N^2}\right) }\, dm \right| \leq \\
                & \leq \left(e^{ \varepsilon \mu} - 1\right) 
                  \int_{\circle \setminus E} \prod_{k=1}^{Q_N} \left(1+ \frac{{2 \langle t , \xi_k \rangle}^2}{\sigma_N^2} \right)^{1/2}
                  \exp{\left(- \frac{{\langle t , \xi_k \rangle}^2}{\sigma_N^2}\right)}\,  dm
                  \leq e^{ \varepsilon \mu} - 1. 
            \end{split}
        \end{equation*}
        The last inequality follows from the elementary estimate $(1+x)^{1/2} e^{-x/2} \leq 1$ if $x \geq 0$.
        Hence
        \begin{equation*}
            \left|\varphi_N (t)
            - \ic \prod_{k=1}^{Q_N} \left(1+ \frac{ i \sqrt{2} \langle t , \xi_k \rangle}{\sigma_N}\right)
                  \exp{ \left(- \frac{{\langle t , \xi_k \rangle}^2}{\sigma_N^2}\right)}\, dm  \right|
            \leq 2m(E) + e^{ \varepsilon \mu} - 1. 
        \end{equation*}
        Taking $\mu = 1/\sqrt{\varepsilon} >1$ we get that both $|t|^2/\mu$ and $\varepsilon \mu$ are small.
        Then estimate \eqref{E} yields the approximation \eqref{eq:FourierTransformApprox} as $N \to \infty.$
        Therefore to prove \eqref{eq:FourierTransform} it is sufficient to show that for any $t \in \C$ one has
        \begin{equation*}
            \lim_{N \to \infty} \ic \prod_{k=1}^{Q_N} \left(1+ \frac{ i \sqrt{2} \langle t , \xi_k \rangle}{\sigma_N}\right)
                                    \exp{\left(- \frac{{\langle t , \xi_k \rangle}^2}{\sigma_N^2}\right)}\, dm 
            = \exp{(-|t|^2 / 2)}. 
        \end{equation*}
        This will follow from Lemma \ref{lemma:IntegralLimit} applied to the functions
        \begin{equation*}
            f_N = \prod_{k=1}^{Q_N} \left( 1+ \frac{ i \sqrt{2} \langle t , \xi_k \rangle}{\sigma_N} \right), 
        \end{equation*}
        \begin{equation*}
            g_N = \frac{1}{\sigma_N^2} \sum_{k=1}^{Q_N} {  \langle t , \xi_k \rangle}^2 - \frac{|t|^2}{2} . 
        \end{equation*}
        According to Lemma \ref{lemma:IntegralLimit} it is sufficient to show
        \begin{equation}
            \label{eq:FunctionFL2Norm}
            \sup_N \|f_N \|_2 < \infty, 
        \end{equation}
        \begin{equation}
            \label{eq:FunctionGL2NormLimit}
            \lim_{N \to \infty}  \|g_N\|_2 =0, 
        \end{equation}
        \begin{equation}
            \label{eq:FunctionFExpectation}
            \lim_{N \to \infty} \ic f_N dm = 1. 
        \end{equation}

        \textit{2. The norm $\norm{f_N}{2}.$}
        Observe that
        \begin{equation*}
            \prod_{k=1}^{Q_N} \left(1 + \frac{{2 \langle t , \xi_k \rangle}^2}{\sigma_N^2} \right)
            = 1 + \sum_{k=1}^{Q_N} \frac{2^{k}}{\sigma_N^{2k}}
                                   \sum {\langle t , \xi_{j_1} \rangle}^2 \ldots {\langle t , \xi_{j_k} \rangle}^2, 
        \end{equation*}
        where the last sum is taken over all collections of indices $1\leq j_1 < \ldots < j_k \leq Q_N$.
        Since ${  \langle t , \xi_n \rangle}^2 \leq |t|^2 |\xi_n|^2$,
        Theorem \ref{thm:SquaredModuliPartialSums},
        part~\eqref{thm:item:L2NormEstimate} of Theorem~\ref{thm:L2NormComparability}
        and estimate~\eqref{eq:NormOfLargeBlock} give that
        \begin{equation*}
            \ic {\langle t , \xi_{j_1} \rangle}^2 \ldots {\langle t , \xi_{j_k} \rangle}^2\, dm
            \leq C(f)^k |t|^{2k} \varphi(N)^{k/8} \sigma_N^{2k},\quad 1 \leq k \leq Q_N.  
        \end{equation*}
        Since the total number of distinct collections of indices $j_1, \ldots , j_k$
        verifying $1\leq j_1 < \ldots < j_k \leq Q_N$ is $\binom{Q_N}{k}$, we deduce
        \begin{equation*}
            \ic \prod_{k=1}^{Q_N} \left(1 + \frac{2{  \langle t , \xi_k \rangle}^2}{{\sigma_N}^2} \right)\, dm
            \leq 1 + \sum_{k=1}^{Q_N} {\binom{Q_N}{k}} \frac{C(f)^k 2^k |t|^{2k} \varphi(N)^{k/8} \sigma_N^{2k}}{\sigma_N^{2k}}. 
        \end{equation*}
        Simplifying last expression we deduce
        \begin{equation*}
            \ic \prod_{k=1}^{Q_N} \left(1 + \frac{{  2 \langle t , \xi_k \rangle}^2}{  \sigma_N^2} \right)\, dm
            \leq \left(1 + 2 C(f) |t|^2 \varphi(N)^{1/8} \right)^{Q_N}.
        \end{equation*}
        Now~\eqref{eq:NumberOfLargeBlocks} implies that $Q_N \lesssim \varphi(N)^{-1/8},$ if $N$ is sufficiently large, 
        from which we get that
        $\norm{f_N}{2}^2 \leq \exp{(3C(f) |t|^2)}.$
        This gives \eqref{eq:FunctionFL2Norm}.

        \textit{3. The norm $\| g_N \|_2.$}
        Recall that
        \begin{equation}
            \label{eq:XiBlockNotation}
            \xi_k = \sum_{n \in {\A}(k)} a_n f^n , \qquad k=1, \ldots , Q_N.
        \end{equation}
        Let $\A = \bigcup_{k=1}^{Q_N} {\A}(k).$
        Observe that~\eqref{eq:SumOfShortBlocks} together with Theorem~\ref{thm:L2NormComparability} gives
        \begin{equation*}
            \lim_{N \to \infty} \frac{\sum_{n \in \A} |a_n|^2 }{S_N^2} = 1. 
        \end{equation*}
        This is assumption \eqref{eq:PVFirst} of Lemma \ref{lemma:PartialVariances}.
        Assumption \eqref{eq:PVSec} is an immediate consequence of~\eqref{eq:QuantitativeGrowthCondition} and \eqref{eq:NormOfLargeBlock}. 
        % \eqref{eq:StrongCoefficientGrowth}
        % (which followed from~\eqref{eq:CoefficientsL2Divergence} and~\eqref{eq:CoefficientGrowth}).
        Thus, Lemma \ref{lemma:PartialVariances} gives 
        \begin{equation}
            \label{eq:XiL2NormSum}
            \lim_{N \to \infty} \frac{\sum_{k=1}^{Q_N} \norm{\xi_k}{2}^2 }{\sigma_N^2} = 1. 
        \end{equation}
        Denote $\lambda = t/|t|$.
        We have
        \begin{equation*}
            g_N =
            \frac{|t|^2}{4 \sigma_N^2} \sum_{k=1}^{Q_N} \left( 2|\xi_k|^2 + \overline{\lambda^2} \xi_k^2 + \lambda^2 \overline{\xi_k^2} - 2 \frac{\sigma_N^2}{Q_N} \right). 
        \end{equation*}
        Applying \eqref{eq:XiL2NormSum}, the proof of \eqref{eq:FunctionGL2NormLimit} reduces to show
        \begin{equation*}
            \lim_{N \to \infty} \norm{\frac{1}{\sigma_N^2} \sum_{k=1}^{Q_N} \psi_k}{2} = 0, 
        \end{equation*}
        where $\psi_k = 2(|\xi_k|^2 -\norm{\xi_k}{2}^2)+ \overline{\lambda^2} \xi_k^2 + \lambda^2 \overline{\xi_k^2}.$
        Now
        \begin{equation}
            \label{eq:PsiSumL2Norm}
            \norm{\sum_{k=1}^{Q_N} \psi_k}{2}^2
            = \sum_{k=1}^{Q_N} \norm{\psi_k}{2}^2 + 2 \Real \sum_{k=1}^{Q_N - 1} \sum_{j>k}^{Q_N} \ic \overline{\psi_k} \psi_j\, dm.
        \end{equation}
        Since $|\psi_k| \leq 4|\xi_k|^2 + 2 \norm{\xi_k}{2}^2$,
        parts \eqref{thm:item:L2NormEstimate} and \eqref{thm:item:L4NormEstimate}
        of Theorem~\ref{thm:L2NormComparability} and estimate~\eqref{eq:NormOfLargeBlock}
        give that $\norm{\psi_k}{2}^2 \leq C(f) \varphi(N)^{1/4} \sigma_N^4.$
        Hence, using that $Q_N \simeq \varphi(N)^{-1/8}$ by expression~\eqref{eq:NumberOfLargeBlocks}, we get
        \begin{equation*}
            \sum_{k=1}^{Q_N} \norm{\psi_k}{2}^2 \lesssim \varphi(N)^{1/8} \sigma_N^4
        \end{equation*}
        and we deduce
        \begin{equation*}
            \lim_{N \to \infty} \frac{1}{\sigma_N^4} \sum_{k=1}^{Q_N} \norm{\psi_k}{2}^2 = 0. 
        \end{equation*}
        The second term in~\eqref{eq:PsiSumL2Norm} splits as
        \begin{equation*}
            \sum_{k=1}^{Q_N - 1} \sum_{j>k}^{Q_N} \ic \overline{\psi_k} \psi_j\, dm = A + B + C + D, 
        \end{equation*}
        where 
        \begin{equation*}
            A = 4 \sum_{k=1}^{Q_N - 1} \sum_{j>k}^{Q_N}
            \ic \left(|\xi_k|^2 - \norm{\xi_k}{2}^2\right)\left(|\xi_j|^2 - \norm{\xi_j}{2}^2\right)\, dm , 
        \end{equation*}
        \begin{equation*}
            B = 2 \sum_{k=1}^{Q_N - 1} \sum_{j>k}^{Q_N}
            \ic\left(|\xi_k|^2 - \norm{\xi_k}{2}^2\right)\left(\overline{\lambda^2}\xi_j^2 + \lambda^2 \overline{\xi_j^2}\right)\, dm , 
        \end{equation*}
        \begin{equation*}
            C = 2 \sum_{k=1}^{Q_N - 1} \sum_{j>k}^{Q_N}
            \ic\left(\overline{\lambda^2} \xi_k^2 + \lambda^2 \overline{\xi_k^2}\right)\left(|\xi_j|^2 - \norm{\xi_j}{2}^2\right)\, dm ,
        \end{equation*}
        \begin{equation*}
            D = \sum_{k=1}^{Q_N - 1} \sum_{j>k}^{Q_N}
            \ic\left(\overline{\lambda^2} \xi_k^2 + \lambda^2 \overline{\xi_k^2}\right)
               \left(\overline{\lambda^2} \xi_j^2 + \lambda^2 \overline{\xi_j^2}\right)\, dm.  
        \end{equation*}
By Theorem \ref{thm:SquaredModuliPartialSums},
        $\norm{\xi_k \xi_j}{2} = \norm{\xi_k}{2} \norm{\xi_j}{2}$ if $k \neq j$ and we deduce $A=0.$
        Since the mean of $\xi_j^2$ over the unit circle vanishes and at almost every point in the unit circle one has  
        \begin{equation}
            \label{eq:XiModulusSquaredExpression}
            |\xi_k|^2 = \sum_{n \in {\A}(k)} |a_n|^2 + 2 \Real \sum_{n \in {\A}(k)}\sum_{j \in {\A}(k), j>n} \overline{a_n} a_j \overline{f^n} f^j, 
        \end{equation}
        the integrals in $B$ can be written as a linear combination of 
        \begin{equation*}
            \ic f^{n_1} \overline{f^{j_1}} \left(\overline{\lambda^2} \xi_j^2 + \lambda^2 \overline{\xi_j^2}\right)\, dm,
        \end{equation*}
        where $n_1, j_1 \in {\A}(k)$ and
        hence $\max \{n_1, j_1\}< \min \{n\colon n \in {\A}(j)\}.$
        According to part \eqref{lemma:item:FourFactorsCancellation} of Lemma \ref{lemma:FourFactors}, 
        \begin{equation*}
            \ic f^{n_1} \overline{f^{j_1}} \xi_j^2\, dm =0 
        \end{equation*}
        and we deduce $B=0$. Since the mean of $\xi_k^2$ over the unit circle vanishes, we have
        \begin{equation*}
            C= 4  \Real \sum_{k=1}^{Q_N - 1} \overline{{\lambda}^2}  \sum_{j=k+1}^{Q_N} \ic \xi_k^2 |\xi_j|^2\, dm . 
        \end{equation*}
        Fix $j>k$. Using again that $\xi_k$ vanishes at the origin and formula \eqref{eq:XiModulusSquaredExpression}, we have 
        \begin{equation*}
            \ic \xi_k^2 |\xi_j|^2\, dm = \ic \xi_k^2 2 \Real \sum_{r, l \in {\A (j)}, l>r} \overline{a_r} a_l \overline{f^r} f^l \,  dm.  
        \end{equation*}
       % where
        %\begin{equation*}
         %   h_j= 2  \sum_{r, l \in {\A (j)}, l>r} \overline{a_r} a_l \overline{f^r} f^l %. 
        %\end{equation*}
        Using formula \eqref{eq:XiBlockNotation} to expand $\xi_k^2$, we obtain 
        \begin{equation*}
            \ic \xi_k^2 |\xi_j|^2\, dm = E +F, 
        \end{equation*}
        where
        \begin{equation*}
            E =
            \sum_{n \in {\A (k)}}  \sum_{r, l \in {\A (j)}, l>r} a_n^2
                \ic (f^n)^2 \left( \overline{a_r} a_l \overline{f^r} f^l + a_r \overline{a_l} f^r  \overline{f^l} \right)\, dm,  
        \end{equation*}
        \begin{equation*}
            F =
            2 \sum_{n,s \in {\A (k)} \colon s>n} a_n a_s \sum_{r, l \in {\A (j)}, l>r} 
                \ic f^n f^s \left( \overline{a_r} a_l \overline{f^r} f^l + a_r \overline{a_l} f^r  \overline{f^l}\right)\, dm.   
        \end{equation*}
        By part \eqref{lemma:item:FirstFactorSquared} of Lemma \ref{lemma:FourFactors} we have
        \begin{equation*}
            \left|\ic (f^n)^2 \overline{f^r} f^l\, dm\right| +
                \left|\ic (f^n)^2 \overline{f^l} f^r\, dm\right|
            \leq C(f) |f'(0)|^{l-n},  \quad \text{ if } n < r < l, r-n \geq 3.   
        \end{equation*}
        We deduce that 
        \begin{equation*}
            |E| \leq C(f) \sum_{n \in \A (k)} |a_n|^2 \sum_{r, l \in \A (j), l>r}  |a_r| |a_l| |f'(0)|^{l-n} . 
        \end{equation*}
         Denote $ \psi (N) = \varphi(N)^{-1/2}/2$. 
        According to~\eqref{eq:NumberOfTermsPerBlock}, since $j>k$ 
        we have $r - n \geq \psi (N)$ for any $r \in {\A}(j)$ and any $n \in {\A}(k).$ Hence we have 
        \begin{equation}
        \label{nou}
            \sum_{r, l \in {\A}(j), l>r}  |a_r| |a_l| |f'(0)|^{l-n}
            \leq |f'(0)|^{ \psi (N)} \sum_{t \geq 1} |f'(0)|^{t} \sum_{r \in {\A}(j)\colon r+t \in {\A}(j)}  |a_r| |a_{r+t}|. 
        \end{equation}
        By Cauchy-Schwarz's inequality, the last sum is bounded by $\sum_{r \in {\A}(j)} |a_r|^2 \lesssim \varphi(N)^{1/8} \sigma_N^2.$
        Hence
        \begin{equation}
            \label{eq:EstimateE}
            |E| \leq C(f) |f'(0)|^{\psi (N)} \varphi(N)^{1/4} \sigma_N^4.
        \end{equation}
        Similarly, part \eqref{lemma:item:FourArbitraryFactors} of Lemma \ref{lemma:FourFactors} gives that 
        \begin{equation*}
            \left|\ic f^n f^s \overline{f^r} f^l\, dm \right|
                + \left|\ic f^n f^s f^r \overline{f^l}\, dm\right|
            \leq C(f) |f'(0)|^{l-n}, \quad n<s<r<l, 
        \end{equation*}
        if $r-s \geq 3$. 
        % and
        % \begin{equation*}
        %     \left|\ic f^n f^s \overline{f^r} f^l\, dm\right|
        %         + \left|\ic f^n f^s f^r \overline{f^l}\, dm\right|
        %     \leq C(f) |f'(0)|^{r-n}, \quad n<s<r<l, r\geq l-2 . 
        % \end{equation*}
       % Using the trivial estimate $|a_k| \leq S_N$, $k \leq N$, we deduce that 
        % \begin{equation*}
        %     |F| \leq C(f) S_N^4 \sum_{n,s \in {\A}(k)\colon s>n} \quad \sum_{r, l \in {\A}% ( j), l>r} |f'(0)|^{l-n}. 
        % \end{equation*}
Then 
\begin{equation*}
            |F| \leq 2 C(f) \sum_{n,s \in {\A}(k)\colon s>n} |a_n| |a_s| \quad \sum_{r, l \in {\A}(j), l>r} |f'(0)|^{l-n} |a_r| |a_l|. 
        \end{equation*}
        As before, since $j>k$ we have that $r - s \geq \psi (N)$ for any $r \in {\A}(j)$ and any $s \in {\A}(k).$ Hence $l-n = l-r+ r-s+s-n \geq \psi (N) + l-r + s-n$, for any $n,s \in {\A}(k)$, $s>n$ and any $r, l \in {\A}(j)$, $l>r$. Hence
        $$
        |F| \leq  2 C(f) |f'(0)|^{\psi (N)} \sum_{n,s \in {\A}(k)\colon s>n} |f'(0)|^{s-n} |a_n| |a_s| \quad \sum_{r, l \in {\A}(j), l>r} |f'(0)|^{l-r} |a_r| |a_l|.
        $$
        Repeating the argument in \eqref{nou} we obtain
        \begin{equation}
            \label{eq:EstimateF}
            |F| \leq 2 C(f) \sigma_N^4 |f'(0)|^{\psi (N)} \varphi (N)^{1/4}.
        \end{equation}
        Now, the exponential decay in~\eqref{eq:EstimateE} and~\eqref{eq:EstimateF} give that 
        \begin{equation}
            \lim_{N \to \infty} \frac{C}{\sigma_N^4} = 0. 
        \end{equation}
        The corresponding estimate for $D$ follows from the estimate
        \begin{equation*}
            \left|\ic {\xi_k}^2 \overline{{\xi_j}^2}\, dm\right|
            \leq C(f) S_N^4 |f'(0)|^{\psi (N)}, \quad k<j. 
        \end{equation*}
        As before this last estimate follows from~\eqref{eq:NumberOfTermsPerBlock} and from
        \begin{equation*}
            \left|\ic f^n f^s \overline{f^l}\overline{f^t}\, dm\right|
            \leq C(f) |f'(0)|^{t-n}, n \leq s < l \leq t, l-s \geq 3,
        \end{equation*}
which follows from parts \eqref{lemma:item:FourArbitraryFactors} of Lemma \ref{lemma:FourFactors} when $n<s<l<t$, part (c) of Lemma \ref{lemma:FourFactors} if $n=s, l<t$ and part (d) if $l=t$.
        This finishes the proof of \eqref{eq:FunctionGL2NormLimit}.

        \textit{4. The integral of $f_N.$}
        In this last step we will prove \eqref{eq:FunctionFExpectation}.
        Observe that at almost every point in the unit circle we have 
        \begin{equation}
        \label{eq:lesfn}
            f_N = 1+ \sum_{k=1}^{Q_N} \frac{i^k 2^{k/2}}{\sigma_N^k}
                \sum  \langle t , \xi_{i_1} \rangle \ldots  \langle t , \xi_{i_k} \rangle ,  
        \end{equation}
        where the second sum is taken over all collections of indices $1 \leq i_1 < \ldots < i_k \leq Q_N $.
        Fix $1 \leq i_1 < \ldots < i_k \leq Q_N $.
        The integral
        \begin{equation*}
            \ic \langle t , \xi_{i_1} \rangle \ldots  \langle t , \xi_{i_k} \rangle\, dm
            = 2^{-k} \ic \prod_{n=1}^k \left(\overline{t} \xi_{i_n} + t \overline{ \xi_{i_n}}\right)\, dm
        \end{equation*}
        is a multiple of a sum of $2^k$ integrals of the form
        \begin{equation*}
            {\overline{t}}^r t^l \ic  \xi_{i_1}^{\varepsilon_1} \ldots  \xi_{i_k}^{\varepsilon_k}\, dm ,
        \end{equation*}
        where $r+l = k$, $\varepsilon_i =1  $ or $\varepsilon_i = -1  $ for $i=1, \ldots , k$. We recall the notation ${\xi_i}^{-1} (z) = \overline{\xi_i (z)} $, $ z \in \circle$.
        Now, each $\xi_i$ is a linear combination of iterates of $f$, that is
        \begin{equation*}
            \xi_j = \sum_{n \in \A (j)} a_n f^n . 
        \end{equation*}
        Hence
        \begin{equation*}
            \ic \xi_{i_1}^{\varepsilon_1} \ldots  \xi_{i_k}^{\varepsilon_k}\, dm
            = \sum_{\boldsymbol{n} \in \mathcal{C}} \prod_{j=1}^k a_{n_j}^{\varepsilon_j} \ic f^{n_1 \varepsilon_1} \ldots f^{n_k \varepsilon_k}\, dm, 
        \end{equation*}
        where $\sum_{\boldsymbol{n} \in \mathcal{C}}$ means the sum over the collection $\mathcal{C}$ of all possible $k$-tuples $\boldsymbol{n} = \{n_j\}_{j=1}^k$ of indices
        such that $n_j \in \A (i_j)$ for $j = 1, \ldots, k.$
        Since $|a_n| \leq S_N$, $n \leq N$, we have
        \begin{equation}
        \label{eq:integraldelesxin}
            \left|\ic \xi_{i_1}^{\varepsilon_1} \ldots  \xi_{i_k}^{\varepsilon_k}\, dm\right|
            \leq S_N^k \sum_{\boldsymbol{n} \in \mathcal{C}} \left| \ic f^{n_1 \varepsilon_1} \ldots f^{n_k \varepsilon_k}\, dm\right|.  
        \end{equation}
        Let $\boldsymbol{\varepsilon} = \{\varepsilon_j\}_{j=1}^{k}$ be fixed and consider
        $\Phi (\boldsymbol{n}) = \Phi(\boldsymbol{\varepsilon},\boldsymbol{n}) = \sum_{j=1}^{k-1}  \delta_j (n_{j+1} - n_j)$
        where $\delta_j \in \{0,1/2,1\}$ for $j = 1, \ldots, k-1,$
        with $\delta_1 = 1$ and $\delta_{k-1} \geq 1/2,$
        and with $\delta_j = 1$ if and only if $\delta_{j-1} = 0$ for $j = 2, \ldots, k-1,$
        as introduced in Theorem \ref{thm:HigherCorrelations}.
        Let $a = |f'(0)|.$
        Theorem~\ref{thm:HigherCorrelations} gives 
        \begin{equation}
        \label{cotaintegraljin}
            \left|\ic \xi_{i_1}^{\varepsilon_1} \ldots  \xi_{i_k}^{\varepsilon_k}\, dm\right|
            \leq k! S_N^k C(f)^k \sum_{\boldsymbol{n} \in \mathcal{C}} a^{\Phi(\boldsymbol{n})}.
        \end{equation}
        We split the sum over $\boldsymbol{n} \in \mathcal{C}$ as follows.
        Let $\mathcal{D}$ denote the set of $(k-1)$-tuples $\boldsymbol{\delta} = \{\delta_j\}_{j=1}^{k-1}$
        of coefficients that can appear in $\Phi(\boldsymbol{n})$ as one varies $\boldsymbol{n},$
        that is, those tuples with $\delta_j \in \{0,1/2,1\}$ for $j = 1, \ldots, k-1,$
        with $\delta_1 = 1$ and $\delta_{k-1} \geq 1/2,$
        and with $\delta_j = 1$ if and only if $\delta_{j-1} = 0,$ for $j = 2, \ldots, k-1.$
        Observe that $\mathcal{D}$ has at most $2^k$ elements.
        Given a $k$-tuple $\boldsymbol{n} \in \mathcal{C},$
        let us denote by $\boldsymbol{\delta}(\boldsymbol{n})$ the $(k-1)$-tuple $\boldsymbol{\delta}$
        of coefficients appearing in $\Phi(\boldsymbol{n})$ for that particular value of $\boldsymbol{n}.$
        Then we have that
        \begin{equation*}
            \sum_{\boldsymbol{n} \in \mathcal{C}} a^{\Phi(\boldsymbol{n})} =
            \sum_{\boldsymbol{\delta} \in \mathcal{D}} \sum_{\{\boldsymbol{n} \in \mathcal{C}\colon \boldsymbol{\delta}(\boldsymbol{n}) =
            \boldsymbol{\delta}\}} a^{\Phi(\boldsymbol{n})}.
        \end{equation*}
        Given $\boldsymbol{\delta} = \{\delta_j\}_{j=1}^{k-1} \in \mathcal{D}$ fixed,
        we define $\Phi_{\boldsymbol{\delta}}(\boldsymbol{n}) = \sum_{j=1}^{k-1} \delta_j (n_{j+1} - n_j)$
        for every $\boldsymbol{n} \in \mathcal{C}.$
        We clearly have that
        \begin{equation}
            \label{eq:SumFixingDelta}
            \sum_{\boldsymbol{n} \in \mathcal{C}} a^{\Phi(\boldsymbol{n})} \leq
            \sum_{\boldsymbol{\delta} \in \mathcal{D}} \sum_{\boldsymbol{n} \in \mathcal{C}} a^{\Phi_{\boldsymbol{\delta}}(\boldsymbol{n})}.
        \end{equation}
        Consider now a fixed $\boldsymbol{\delta} = (\delta_1,\ldots,\delta_{k-1}) \in \mathcal{D}$
        and let us compute the rightmost sum in~\eqref{eq:SumFixingDelta}.
        Recall that $\delta_1 = 1.$
        It may be useful for the following argument to have in mind that $\boldsymbol{\delta}$
        might look like $\boldsymbol{\delta}= (1, 1/2, \ldots, 1/2, 0,1,0,1, 1/2, \dots, ,1/2).$
        This is to say that $\boldsymbol{\delta}$ has consecutive chains of nonzero elements
        (at least one such chain),
        all of them starting by $1$ and with the rest of the terms, if any, equal to $1/2.$
        Moreover, there is exactly one null term between each consecutive pair of such chains.
        Let $l(1)$ be the minimum integer such that $\delta_{l(1)+1} = 0$
        (we set $l(1) = k-1$ if $\delta_j \neq 0$ for all $1 \leq j \leq k-1$).
        This is to say that $l(1)$ is the last index of the first chain of nonzero elements.
        In particular, observe that if $l(1) > 1,$ we have that $\delta_j = 1/2$ for $2 \leq j \leq l(1)$
        because by Theorem \ref{thm:HigherCorrelations} $\delta_j =1$ if and only if $\delta_{j-1} =0$ for $j>1$.
        Assume now that we have determined $l(m-1).$
        If $l(m-1) < k-1,$ then let $l(m)$ be the minimum integer such that $l(m-1) < l(m) \leq k-1$
        and such that $\delta_{l(m)+1} = 0.$
        As before, we set $l(m) = k-1$ if $\delta_j \neq 0$ for all $l (m-1)+2 \leq j \leq k-1$. We iterate this process until we set $l(r) = k-1$ for some integer $1 \leq r \leq k.$
        Roughly speaking the indices $l(m)$ indicate the end of the strips of nonzero terms in the components of $\boldsymbol{\delta} $.
        Taking $l(0) = -1,$ note that 
        $$
        \Phi_{\boldsymbol{\delta}}(\boldsymbol{n}) = \sum_{j=1}^{k-1} \delta_j (n_{j+1} - n_j) = \sum_{m=1}^r \left(( n_{l(m-1) +3} -  n_{l(m-1) +2}) + \frac12 (n_{l(m)+1} - n_{l(m-1) +3})\right). 
        $$
        Hence the sum 
        $$
         \sum_{\boldsymbol{n} \in \mathcal{C}} a^{\Phi_{\boldsymbol{\delta}}(\boldsymbol{n})}
        $$
        in the right-hand side of \eqref{eq:SumFixingDelta}
        becomes a product over $m = 1, \ldots, r$, of sums of the form 
        \begin{equation}
            \label{productesparcials}
            \sum_{\substack{n_j \in \A(i_j)\\ j=l(m-1)+2,\ldots,l(m)+1}} a^{ (n_{l(m-1) +3} -  n_{l(m-1) +2}) + \frac12 (n_{l(m)+1} - n_{l(m-1) +3})} .
        \end{equation}
        Thus, to estimate 
        \begin{equation*}
            \sum_{\boldsymbol{n} \in \mathcal{C}} a^{\Phi_{\boldsymbol{\delta}}(\boldsymbol{n})},
        \end{equation*}
        we need to estimate sums of the form \eqref{productesparcials}. 
       % \begin{equation}
        %    \label{eq:PartialSumOverIndices}
        %    \sum_{j=l(m-1)+1}^{l(m)} \sum_{n_j \in \A(i_j)} a^{(n_2 - n_1) + (n_l - n_2)/2}
        %\end{equation}
        % for $1 < l \leq k-1.$ Fix $l>2$. 
        Since the argument is identical for any $m$, we present it for $m=1$. To simplify the notation write $l=l(1)$ and assume $l>2.$ Denote here $\overline{n_1} = \max \A(i_1),$ $\underline{n_2} = \min \A(i_2)$ and $\underline{n_{l+1}} = \min \A(i_{l+1}),$
        and observe that $\underline{n_2} - \overline{n_1} \geq \psi (N)$ because of~\eqref{eq:NumberOfTermsPerBlock}. Here as before, $\psi (N) = \varphi (N)^{-1/2} / 2$. Notice that the exponent in \eqref{productesparcials} is 
        $$
        (n_{2} -  n_{1}) + \frac12 (n_{l+1} - n_{2}) = \frac12 (n_{l+1} - n_1) + \frac12 (n_2 - n_1) \geq \frac12 \psi (N) + \frac12 (n_{l+1} - n_1).
        $$
        Summing over $n_1$ and $n_2$ we get that \eqref{productesparcials} is bounded by
        \begin{equation*}
            C^2 a^{\psi (N) / 2} \sum_{\substack{n_j \in \A(i_j)\\ j=3,\ldots,l+1}} a^{(n_{l+1} - \overline{n_1})/2}.
        \end{equation*}
        Next, summing over $n_j$ for $j$ up to $l$ yields the factor $|\A(i_3)| \cdot |\A(i_4)| \cdots |\A(i_{l})|,$
        while summing over $n_{l+1}$ we get the factor $a^{(\underline{n_{l+1}} - \overline{n_1})/2}.$ In other words,
        \begin{equation*}
            \sum_{\substack{n_j \in \A(i_j)\\ j=3,\ldots,l+1}} a^{(n_{l+1} - \overline{n_1})/2} \leq C^3 a^{(\underline{n_{l+1}} - \overline{n_1})/2} (|\A(i_3)| \cdots |\A(i_{l})|). 
        \end{equation*}
        Here, $|\A(i_j)|$ denotes the number of indices in the set $\A(i_j).$
        Recall that between each consecutive pair of sets $\A(i_j),$
        there is a set $\B(i_j)$ with at least $\psi(N)$ terms.
        Thus, we have that $\underline{n_{l+1}} - \overline{n_1} \geq l\psi (N) + |\A(i_3)| + \ldots + |\A(i_{l})|.$
        Hence, we find that
        \begin{equation}
            \label{eq:BoundForPartialSumOverIndices}
            \sum_{\substack{n_j \in \A(i_j)\\ j=1,\ldots,l+1}} a^{(n_2 - n_1) + (n_{l+1} - n_2)/2} \leq C^{l+1} a^{(l+1)\psi (N) / 2}
        \end{equation}
        because each factor $a^{|\A(i_j)|} |\A(i_j)|$ is bounded by a universal constant.
        Note that if $l =1$ or $l = 2$, then \eqref{eq:BoundForPartialSumOverIndices} is obvious.

        Now,the full sum 
        \begin{equation*}
            \sum_{\boldsymbol{n} \in \mathcal{C}} a^{\Phi_{\boldsymbol{\delta}}(\boldsymbol{n})},
        \end{equation*}
        in the right-hand side of \eqref{eq:SumFixingDelta}
        becomes a product over $m = 1, \ldots, r$, of sums of the form \eqref{productesparcials}.
        Thus, applying the estimate \eqref{eq:BoundForPartialSumOverIndices} we get that
        \begin{equation*}
            \sum_{\boldsymbol{n} \in \mathcal{C}} a^{\Phi_{\boldsymbol{\delta}}(\boldsymbol{n})} \leq
            \prod_{m=1}^r (C a^{\psi (N) / 2})^{(l(m)-l(m-1))} \leq
            C^k a^{k\psi (N) / 2}.
        \end{equation*}
        Next, summing over $\boldsymbol{\delta} \in \mathcal{D}$ and using the fact that $\mathcal{D}$ has at most $2^k$ such tuples, we get that the sum in the left hand side of   \eqref{eq:SumFixingDelta} can be estimated as 
        \begin{equation*}
            \sum_{\boldsymbol{n} \in \mathcal{C}} a^{\Phi(\boldsymbol{n})} \leq (2C)^k a^{k \psi (N) / 2}.
        \end{equation*}
        Thus, using this in~\eqref{cotaintegraljin}, there exists a constant $C(f)>0$ such that 
        \begin{equation*}
            \left|\ic \xi_{i_1}^{\varepsilon_1} \ldots  \xi_{i_k}^{\varepsilon_k}\, dm\right|
            \leq k! S_N^k C(f)^k a^{k \psi (N) / 2}.
        \end{equation*}
        We deduce that 
        \begin{equation*}
            \left|\ic \langle t , \xi_{i_1} \rangle \ldots  \langle t , \xi_{i_k} \rangle\, dm\right|
            \leq k! S_N^k C(f)^k |t|^k a^{k \psi (N) / 2}.
        \end{equation*}
        Since the total number of collections of indices $1 \leq i_1 < \ldots < i_k \leq Q_N$ is $\binom{Q_N}{k}$, using \eqref{eq:lesfn}, it follows that
        \begin{equation*}
            \left|\ic f_N\, dm - 1\right|
            \leq \sum_{k=1}^{Q_N} {\binom{Q_N}{k}} k! 2^{k/2} \sigma_N^{-k} (C(f)S_N |t|)^k a^{k \psi (N) / 2}.
        \end{equation*}
        Finally, using that $k! \leq Q_N^k$ for any integer $1 \leq k \leq Q_N$, last sum is smaller than 
        \begin{equation*}
            \left( 1 + \frac{C(f)\sqrt{2}|t| S_N Q_N  a^{\psi (N) / 2}}{\sigma_N} \right)^{Q_N} - 1,
        \end{equation*}
        which tends to $0$ as $ N \to \infty$ because by~\eqref{eq:NumberOfLargeBlocks} we have
        \begin{equation*}
            \lim_{N \to \infty} \frac{S_N Q_N^2 a^{\psi (N) / 2}}{\sigma_N} = 0. 
        \end{equation*}
    \end{proof}
    
    \printbibliography

    \Addresses

\end{document}